\documentclass[10pt]{article}
\usepackage[top=1in,bottom=1in, left=1in, right=1in]{geometry}
\usepackage[colorlinks=true, linkcolor=blue,citecolor=blue, urlcolor=blue]{hyperref}
\usepackage{graphicx}
\usepackage{mathtools}
\usepackage{amsmath}
\usepackage{amsthm}
\usepackage{amssymb}
\usepackage{dsfont}
\usepackage{cleveref}
\usepackage[utf8]{inputenc}
\usepackage{enumerate}
\usepackage{bm}
\usepackage[title]{appendix}
\usepackage{ltablex}

\usepackage{tikz}
\usetikzlibrary{calc}

\DeclareMathAlphabet\mathbfcal{OMS}{cmsy}{b}{n}

\newcommand{\R}{{\mathbb R}}
\newcommand{\N}{{\mathbb N}}
\newcommand{\Poi}[1]{{\mathbfcal{P}(#1)}}

\renewcommand{\d}{{\mathrm d}}
\renewcommand{\div}{{\mathrm{div}}}

\newcommand{\density}[1]{f_{#1}}
\newcommand{\RNderivative}[2]{{\frac{\d{#1}}{\d{#2}}}}

\newcommand{\BBEnergy}{S}
\newcommand{\dom}{D}
\newcommand{\domDelta}{\mathcal{D}}
\newcommand{\domDeltaHalf}{D_{\delta/2}}
\newcommand{\boundDomDelta}{{\partial\domDelta\times\partial\domDelta}}

\newcommand{\momentum}{\eta}
\newcommand{\spacetime}{{[0,T]\times\dom}}

\newcommand{\numParticles}{n}
\newcommand{\numNonScattered}{m}
\newcommand{\numScattered}{N_s}
\newcommand{\numRegionsA}{n}
\newcommand{\numRegionsB}{m}
\newcommand{\numParticlesBayes}{J}

\newcommand{\norm}[1]{\left\lVert#1\right\rVert}

\newcommand{\absNorm}[1]{\left\lvert#1\right\rvert}

\DeclareMathOperator*{\argmin}{argmin} 
\newcommand{\Log}[1]{\log\left(#1\right)}

\newcommand{\baseMeasMeas}{P_{\mathrm{meas}}}
\newcommand{\baseMeasPart}{P_{\mathrm{part}}}
\newcommand{\baseMeasPartCondTimePoints}[1]{P_{\mathrm{part}|#1}}
\newcommand{\baseMeasPartTimePoints}[1]{P_{#1}}
\newcommand{\probMeasPartCondTimePoints}[2]{P_{#1|#2}}

\newcommand{\meas}{{\mathcal M}}
\newcommand{\measp}{{\mathcal M_+}}

\newcommand{\restr}{{\mbox{\LARGE$\llcorner$}}}
\newcommand\funcRestr[2]{{
		\left.\kern-\nulldelimiterspace 
		#1 
		\vphantom{\big|} 
		\right|_{#2} 
}}
\newcommand{\pushforward}[2]{{{#1}_{\#}#2}}
\newcommand{\hd}{\mathcal{H}}
\newcommand{\weakstarto}{{\stackrel*\rightharpoonup}}

\newcommand{\halflife}{{T_{1/2}}}
\newcommand{\halflifeHat}{{\hat{T}_{1/2}}}

\newcommand{\att}{{\mathrm a}}
\newcommand{\sct}{{\mathrm s}}
\newcommand{\dt}{{\mathrm d}}
\newcommand{\Aq}[1]{A^{#1}}

\newcommand{\discrete}{{\mathrm dscrt}}
\newcommand{\cont}{{\mathrm cont}}
\newcommand{\vol}{{\mathrm{vol}}}

\newcommand{\Mv}{\mathrm{Mv}}
\newcommand{\eval}{\mathrm{eval}}
\newcommand{\reconMap}{\mathrm{rec}}
\newcommand{\recon}{R}

\newcommand{\quant}{Q}

\theoremstyle{plain}
\newtheorem{theorem}{Theorem}[section]
\newtheorem{proposition}[theorem]{Proposition}

\newtheorem{lemma}[theorem]{Lemma}
\newtheorem{corollary}[theorem]{Corollary}
\newtheorem{remark}[theorem]{Remark}
\allowdisplaybreaks 

\newcommand{\notinclude}[1]{}

\title{A Bayesian model for dynamic mass reconstruction from PET listmode data}
\author{Marco Mauritz \and Bernhard Schmitzer \and Benedikt Wirth}

\begin{document}
\maketitle

\begin{abstract}
Positron emission tomography (PET) is a classical imaging technique to reconstruct the mass distribution of a radioactive material.
If the mass distribution is static, this essentially leads to inversion of the X-ray transform.
However, if the mass distribution changes temporally, the measurement signals received over time (the so-called listmode data) belong to different spatial configurations.
We suggest and analyse a Bayesian approach to solve this dynamic inverse problem that is based on optimal transport regularization of the temporally changing mass distribution.
Our focus lies on a rigorous derivation of the Bayesian model and the analysis of its properties,
treating both the continuous as well as the discrete (finitely many detectors and time binning) setting.
\end{abstract}

\section{Introduction}
A typical procedure in positron emission tomography (PET) is to inject a radioactive tracer into a patient or a lab animal.
This tracer then binds to molecules, cells, or tissue of interest, and reconstructing the tracer distribution provides information about the distribution of molecules, cells, or tissue.
Another widespread approach is to directly feed leukocytes with radioactive material and to subsequently inject those, then PET allows to follow the leukocyte distribution over time.\\
The principle of PET is as follows:
Each time $t$ a radioactive atom decays, a pair of photons is emitted in opposite directions.
This photon pair is then detected at two locations $a,b\in\partial \domDelta$ by detectors sitting on the boundary $\partial \domDelta$ of the PET scanning device $\domDelta\subset\R^3$.
The PET measurement thus consists of a list $E=(t_k,a_k,b_k)_{k=1,\ldots,K}\in(\R\times\partial \domDelta\times\partial \domDelta)^K$ of such triples,
so-called \emph{listmode data}, indicating that at time $t_k$ a radioactive decay happened on the line segment connecting $a_k$ and $b_k$, the so-called \emph{line of response}.

If the radioactive mass distribution is static, lines of response intersect in exactly those places where the radioactive material sits.
However, if it changes over time, for instance if the radioactively labelled leukocytes travel within the body,
then the lines of response are all induced by decays at different positions, and a temporal regularization becomes necessary to connect the information from the different lines of response.

\subsection{Contribution of the article}
For reconstructing a temporally changing radioactive mass distribution $\rho\in\measp(\spacetime)$
(a nonnegative Radon measure living on the time interval $[0,T]$ and the convex compact PET scanning area $\dom\subset \domDelta$),
in previous work \cite{ScScWi20} we proposed to minimize the functional
\begin{equation}\label{eqn:basicFunctional}
J^{E,q}(\rho,\momentum)=\norm{A\rho}-\int_{[0,T]\times\partial \domDelta\times\partial \domDelta}\log\left(\RNderivative{\Aq{\rho}}{\nu}\right)\d E+\beta\BBEnergy(\rho,\momentum)
\end{equation}
(note that in this notation we interpret the list $E$ as the linear combination $\sum_{k=1}^K\delta_{(t_k,a_k,b_k)}$ of Dirac measures).
Here, $A$ is the linear forward operator, mapping a mass distribution $\rho$ to an expected photon pair intensity on $\R\times\partial \domDelta\times\partial \domDelta$,
and $\RNderivative{\Aq{\rho}}{\nu}$ is the Radon-Nikodym derivative of $A^q$ (which is a particular modification of $A$ with a parameter $q>0$) with respect to a suitable measure $\nu$.  The modification becomes necessary since photons can sometimes be scattered
and thereby lead to incorrect lines of response among the measurements.
The auxiliary variable $\momentum\in\meas(\spacetime)^3$ is an $\R^3$-valued Radon measure.
It has the interpretation of the physical momentum associated with the motion of the mass $\rho$
and therefore satisfies the \emph{continuity equation}
\begin{equation}\label{eqn:continuity}
\partial_t\rho+\div\momentum=0
\end{equation}
on $(0,T)\times\R^3$ in the distributional sense,
which is known \cite[Lemma 1.1.2]{chizat_unbalancedOT} to guarantee the disintegration
\begin{equation}
\rho=\d t\otimes\rho_t
\end{equation}
into the temporal Lebesgue measure $\d t$ and time slices $\rho_t\in\measp(\dom)$.
Finally, the parameter $\beta>0$ is a regularization weight,
and $\BBEnergy$ is the so-called Benamou--Brenier functional \cite{BeBr00}
\begin{equation}
\BBEnergy(\rho,\momentum)=
\begin{cases}
\int_0^T\int_\dom\left(\frac{\d\momentum_t}{\d\rho_t}\right)^2\,\d\rho_t\,\d t&\text{if $\rho\geq0$ and \eqref{eqn:continuity} holds,}\\
\infty&\text{else.}
\end{cases}
\end{equation}
The Benamou--Brenier functional is a dynamic formulation of optimal transport:
Minimizing it under the constraint $\rho_0=\mu$, $\rho_T=\nu$ for two nonnegative measures $\mu,\nu\in\measp(\dom)$ on $\dom$ of equal mass
yields exactly $T^{-1}$ times the squared Wasserstein-2 distance between $\mu$ and $\nu$.

While in \cite{ScScWi20} we mainly described the functional and numerically confirmed its efficacy for reconstructing temporally moving mass distributions,
the aim of the current article is to rigorously derive and analyse the functional.
Our contributions are the following.
\begin{itemize}
\item
In \crefrange{sec:BayesianAnsatz}{sec:posterior} we rigorously derive a functional $\hat J^E$ (of which $J^{E,q}$ will be a modification)
as the negative logarithm of a Bayesian posterior density
and thereby interpret its minimizer as a maximum a postiori (MAP) estimate for the reconstruction.
In contrast to other approaches in the literature we do not restrict the Bayesian approach to the discrete setting, in which detectors and time measurements have a finite resolution,
but directly apply it to the continuum limit, which for modern detector sizes and temporal resolution is a good approximation
and which is more interesting and significant from the viewpoint of resolution independence.
The difficulty here is that the Bayesian approach involves probability densities,
however, on infinite-dimensional spaces there is no canonical probability distribution with respect to which the density can be expressed. To circumvent this, we aim for reconstructing finitely many time marginals only and provide corresponding probability distributions; note, though, that the negative log-posterior $\hat J^E$ will turn out to be independent of those.
The discrete setting can then be obtained with small modifications as a special variant of the continuous setting.
\item
The minimizer of $\hat J^E$ will have the flaw that it misinterprets incorrect lines of response (arising from photon scattering) as correct ones.
Actually, MAP estimates are notorious for such behaviour.
In \crefrange{sec:bias}{sec:biasRemoval} we introduce a remedy by modifying the functional $\hat J^E$ to $J^{E,q}$,
which contains an additional tuning parameter $q>0$.
The functional $J^{E,q}$ is derived as a convex relaxation of a more elaborate MAP estimate of mixed integer type
which in addition to $\rho$ and $\momentum$ also tries to estimate which measurements were produced by photon scattering.
We illustrate in a reduced toy model the influence of parameter the $q$, yielding heuristics for its choice depending either on the detector size or the so-called positron range.
We furthermore justify the introduction of $q$ in \cref{thm:debiasing,thm:pProperties}
by interpreting it as a Lagrange multiplier and relating it to the estimate of measurements from scattered photons.
\item
In \cref{sec:existence} we prove existence of minimizers,
almost surely with respect to the measurement $E$ (which actually is a random variable depending on the ground truth mass distribution $\rho^\dagger\in\measp(\spacetime)$).
\item
In \cref{sec:scaling} we analyse the invariances of the minimization problem under certain parameter changes, which leads to a heuristic for choosing the regularization parameter $\beta$.
Essentially, this is a nondimensionalization of our functional, however, with the complication that the measurement $E$ cannot simply be rescaled (since it is a list of lines of response).
Thus, for a rigorous analysis of the invariances we need to take into account the stochastic nature of the measurements and thus also the functional and its minimizers:
We will show that the law of the minimizers transforms in a specific way under certain parameter scalings.
\item
In \cref{sec:minimizerStructure} we employ the recent result \cite{Bredies_extremePointsBB} to show
that our reconstruction $\rho$ will (almost surely with respect to $E$) represent a finite number of particle trajectories.
\item
In \cref{sec:lifting} we relate our reconstruction approach to the model proposed in \cite{Lee2015} that reconstructs single cell trajectories:
If one extends that model to multiple cell trajectories (which turns it into a complicated mixed integer or combinatorial optimization),
then our functional can be viewed as a convex relaxation.
\end{itemize}

\subsection{Preliminaries and notation}
Let us briefly introduce some notation, part of which we actually already used above.
The Banach space of Radon measures on a compact domain $B$ will be denoted $\meas(B)$ with norm $\norm{\cdot}$, the subset of nonnegative measures by $\measp(B)$.
For two measures $\mu,\nu\in\meas(B)$ with $\mu$ absolutely continuous with respect to $\nu$, the Radon--Nikodym derivative of $\mu$ with respect to $\nu$ is denoted $\RNderivative{\mu}{\nu}$.
The restriction of a measure $\mu$ to some $\mu$-measurable set $S$ is denoted $\mu\restr S$,
and the pushforward of $\mu$ under some $\mu$-measurable map $f$ is denoted $\pushforward{f}\mu$.
By $\mathcal{L}^d$ and $\hd^d$ we denote the $d$-dimensional Lebesgue and Hausdorff measure, where for $d=1$ we may drop the exponent,
and $\delta_a$ denotes the Dirac measure at some point $a$.
Sometimes we will for simplicity also refer to the Lebesgue measure in time by $\d t$.
Furthermore, we will indicate random variables by boldface letters such as $\bm E$
while their realizations have normal font, thus $E=\bm E(\omega)$ for $\omega$ a random element of the standard probability space $(\Omega,\mathcal{F},P)$.
For the densities of such random variables $\bm E$ (or rather their probability distributions) with respect to a base probability measure (that will be fixed in the context),
rather than using Radon--Nikodym derivatives we introduce the specific notation $\density{\bm E}(E)$
and $\density{\bm E}(E|A)$ for the density conditioned on some event $A$
(frequently $A$ will be a specific realization of a random variable, in which case we will just write this realization instead of $A$).
Finally, given a measure $\lambda$, by $\Poi{\lambda}$ we denote the Poisson point process with intensity $\lambda$.
We will only consider $\sigma$-finite intensities on $\R^3$ so that the corresponding Poisson point process is proper and simple and thus can be interpreted as a random set of points
(see \cite{bookPPP_last_penrose, ReissPP} for an introduction into Poisson point processes).

We will further employ the notation $a\lesssim b$ to indicate the existence of an independent constant $c>0$ such that $a\leq cb$
(analogously, $b\gtrsim a$ stands for $a\lesssim b$ and $a\approx b$ for $a\lesssim b$ and $b\lesssim a$).
Finally, we introduce some function spaces. $L^p$, $p\ge1$, denotes the standard $L^p$-space and $C$, $C^1$, $C_c^1$ denotes continuous, continuously differentiable (and compactly supported) functions. $C_\dom([0,T]\times\R^3)$ denotes continuous functions being supported inside $\dom$ at time $0$.
For the reader's convenience below we provide a reference list of further model-specific symbols and quantities frequently used throughout the article.

\setlength\extrarowheight{3pt}
\begin{longtable}{lp{12cm}}
	$A^\att, A^\sct, A^\dt$ & Forward operators describing attenuation, scattering and normal detection. They are either defined on time slices, i.e.\ on $\measp(\dom)$, or on $\measp(\spacetime)$ via $A^{a/c/d}\rho = \d t\otimes A^{a/c/d}\rho_t$, see \cref{sec:forwardOperator}\\
	
	$A, \Aq{q}$ & total forward operator $A = p^\sct A^\sct + p^\dt A^\dt$ and unbiased forward operator $\Aq{q} = qp^\sct A^\sct + p^\dt A^\dt $, see \cref{sec:forwardOperator}, \cref{eqn:finalFunctional}  \\
	
	$\mathcal{C}$ & $\mathcal C=\{(\theta,s)\in S^2\times\R^3\,|\,s\in\pi_{\theta^\perp}(\domDeltaHalf)\}$ with $\pi_{\theta^\perp}$ being the projection onto $\theta^\perp$. The X-Ray (see \cref{definitionXRayTransform}) transform maps onto $L^1(\mathcal C)$\\
	
	$\dom\subset\R^3$ & compact and convex set where the tracer material stays \\
	
	$\domDelta\subset\R^3, \delta$ & compact and convex set such that $\dom\subset\domDelta$ and $\mathrm{dist}(\dom,\partial\domDelta)\ge\delta$ for some $\delta>0$. The detectors are located at the boundary $\partial\domDelta$.\\
	
	$E$, $\absNorm{E}$ & measurement, realization of a Poisson point process $\bm E$ with intensity measure $\frac{1}{\halflife}A\rho^\dagger$. To be interpreted as either a set or equivalently as a discrete empirical measure. $\absNorm{E}$ denotes the number of elements in the set \\
	
	$\eval_t$, $\eval_{t_0,\ldots,t_K}$ & $\eval_{t_0,\ldots,t_K}:C([0,T];D)\to\dom^K, \eval_{t_0,\ldots,t_K}\gamma=(\gamma(t_0),\ldots,\gamma(t_K))$\\
	
	$\density{\Poi{A\rho}}$ & density of the random variable $\Poi{A\rho}$ with respect to a suitable reference measure\\
	
	$G_y(x)$ & smooth compactly supported convolution kernel $G_y:\domDeltaHalf\to[0,\infty)$ ($\mathrm{supp}(G_y)\subset B_{\delta/2}(y)$) describes the probability density of an annihilation of a positron emitted from $y$ with an electron \\
	
	$\Gamma_k\times\Gamma_l\subset\boundDomDelta$, $M$ & $\Gamma_k$, $i=1,\ldots,M$ are the detectors in the discrete setting. For $k\neq l$ we have the detector pairs $\Gamma_k\times\Gamma_l$ where photon pairs are registered\\
	
	$\absNorm{\Gamma_1\times\Gamma_2}$ & measure of the set $\Gamma_1\times\Gamma_2\subset\boundDomDelta$, i.e. $\absNorm{\Gamma_1\times\Gamma_2}=\hd^2\otimes\hd^2(\Gamma_1\times\Gamma_2)$  \\
	
	$\hd^d$ & $d$-dimensional Hausdorff measure \\
	
	$\momentum, \momentum^\dagger\in\meas^3(\spacetime)$         & measures describing the material flux corresponding to the temporal variation of the mass distribution $\rho,\rho^\dagger$     \\ 
	
	$I$ & Identity matrix \\
	
	$\hat J^{E}$ & first reconstruction functional, see (\cref{eqn:DefJhat}) \\
	
	$\bar J^{E,q}$ & reconstruction functional taking into account the bias. Includes Lagrange parameter $q$, see \cref{eqn:max} \\
	
	$J^{E,q}$ & final reconstruction functional, see \cref{eqn:finalFunctional}\\

	$\meas(X)$, $\measp(X)$, $\meas(X)^3$ &  space of (non-negative) Radon measures and three dimensional Radon measures on $X$\\ 
		
	$\absNorm{\cdot}$ & Euclidean norm\\
	
	$P$ & X-ray transform, see \cref{definitionXRayTransform}\\
	
	$P_{\bm E|\bm\theta}$, $P_{\bm E}$, $P_{\bm\theta}$ &(conditional) probability distribution of $\bm E$ (given $\bm \theta$) and of $\theta$ \\
	
	$\baseMeasMeas$, $\baseMeasPart$, $\baseMeasPartCondTimePoints{t_0,\ldots,t_K}$ & base measures on the space of measurements, space of particles and space of particle positions at time points $t_0,\ldots,t_K$ \\
	
	$\Poi{\mu}$ & Poisson point process with intensity measure $\mu$ \\

	$p^\att, p^\sct, p^\dt$ & probabilities for attenuation, scattering and normal detection. It holds $p^\att+p^\sct+p^\dt=1$\\
	
	$\pi_L$	& orthogonal projection onto subspace $L$ \\
	
	$\nu$& $\nu=\dt t\otimes(\hd^2\restr\partial\domDelta)\otimes(\hd^2\restr\partial\domDelta)$. The forward operator $\dt t\otimes A\rho_t$ is absolutely continuous w.r.t. $\nu$. In the discrete case we have $\nu=\sum_{i=1}^N\sum_{j,k=1}^M\delta_{((i-\frac{1}{2})\Delta T, z_j, z_k)}$\\

	$q$, $\Aq{q}$ & Lagrange parameter $q>0$ that weighs the influence of the scatter part of the forward operator. It is $\Aq{q} = qp^\sct A^\sct + p^\dt A^\dt$ (see \cref{eqn:finalFunctional}) \\
	
	$R$ & $R\colon\left(\domDeltaHalf\right)\times S^2\to\boundDomDelta,\;
	R(x,v)=\partial\domDelta\cap(x+\R v)\,,$ is the measurement function that maps a point $x$ (where an annihilation has happened) and a direction $v$ onto the photon pair's detection location. This function is comparable to the classical Radon transform.\\

	$\rho, \rho^\dagger\in\measp(\spacetime)$ &        measures describing tracer/mass distribution in spacetime. $\rho^\dagger$ represents the ground truth distribution     \\   
	
	$S^2$ & Sphere in $\R^3$, i.e.\ $S^2 = \{x\in\R^3 \ | \ \absNorm{x}=1\}$ \\
	
	$\BBEnergy$ & Benamou-Brenier functional \\
	
	$\halflife$ & half-life of the considered radionuclide \\
	
	$[0,T]$ & time interval during which the measurements are taking place\\
	
	$t_1,\ldots,t_K\in[0,T]$, $K$ & time points (containing the time points when a photon pair was detected) for Bayesian inference\\
	
	$\tau_i\subset[0,T]$, $N$ & $\tau_i$, $i=1,\ldots,N$ are the time intervals of the discrete setting\\
	
	$\theta_{t_1,\ldots,t_K}$ & time discrete measure of $\theta\in\measp(C([0,T];D))$, $\theta_{t_1,\ldots,t_K}=\pushforward{\eval_{t_1,\ldots,t_K}}\theta\in\measp(\dom^K)$\\
	
	$\theta^\perp$ & orthogonal complement of $\theta$, i.e.\ $\theta^\perp=\{x \ | \ x\cdot\theta=0\}$ \\

	$Z_v^{\Gamma_k\times\Gamma_l}$ & this is the set of points $\{x\in\domDeltaHalf \ | \ R(x,v)\in\Gamma_k\times\Gamma_l\}$, i.e.\ all points possibly contributing to a detection in detector pair $\Gamma_k\times\Gamma_l$ for a given direction $\nu\in S^2$\\

\end{longtable}

\section{The Bayesian dynamic reconstruction model}\label{sec:reconstructionFunctional}
In this section we derive the negative log-posterior $\hat J^E$ for the reconstruction of a spatiotemporally changing radioactive mass distribution $\rho$.
Actually we will derive the negative log-posterior for a different variable $\theta$ instead of $\rho$.
The description in terms of $\rho$ will then result from a final equivalent reformulation.
In the following we will detail the model and the strategy underlying our Bayesian approach,
after which we provide the (linear) forward operator, the likelihood and the prior distribution to finally arrive at the posterior distribution and the functional $\hat J^E$.
Before, we briefly fix the scanning geometry:
The interior of the PET scanner (the measurement volume) will be denoted by $\dom\subset\R^3$ (a compact and convex domain, see \cref{fig:sketch}). The detections take place on $\partial\domDelta$, where $\dom\subset\domDelta$ with $\text{dist}(\dom,\partial\domDelta)\ge\delta>0$ for some compact and convex set $\domDelta\subset\R^3$ with smooth boundary,
and measurements will be taken over a time interval $[0,T]$.\\

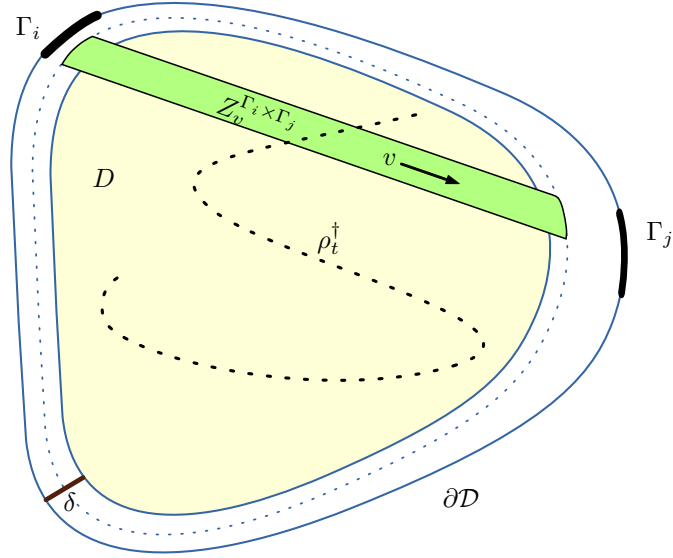
\begin{figure}
	\centering
	\vspace*{\fill} 
	\begin{tikzpicture}
	\node [
	above right,
	inner sep=0] (image) at (3,0) {\includegraphics[width=0.5\textwidth]{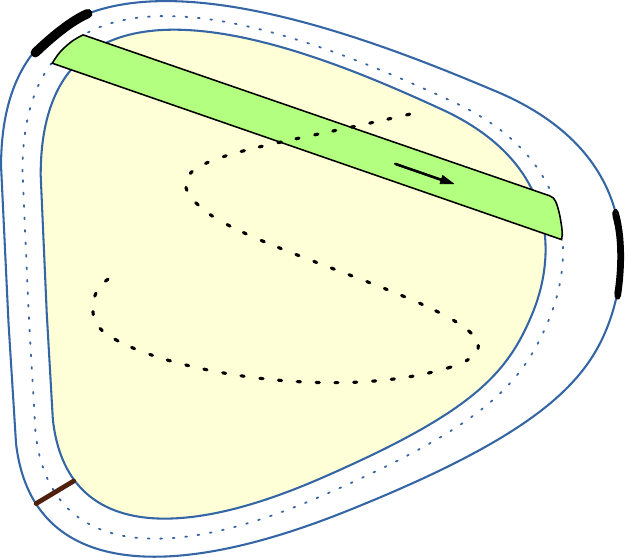}};
	\node at (4.25,5.0){$\dom$};
	\node at (9.,0.75){$\partial\domDelta$};
	\node at (3.25,7.){$\Gamma_i$};
	\node at (11.65,4.25){$\Gamma_j$};
	\node at (7.25,4.2){$\rho_t^\dagger$};
	\node at (3.8,0.7){$\delta$};
	\node at (6.3,5.83){\rotatebox{-21}{$Z_v^{\Gamma_i\times\Gamma_j}$}};
	\node at (8.05,5.25){$v$};
	\end{tikzpicture}\\
	\vspace*{\fill}
	\caption{Two-dimensional sketch of a PET scanning geometry indicating the used notation. The dotted line is the boundary of $\domDeltaHalf$.}
	\label{fig:sketch}
\end{figure}

\begin{remark}[Nonsmooth $\partial\domDelta$]
To improve readability we restrict ourselves to smooth $\partial\domDelta$, even though the extension to arbitrary convex $\domDelta$ is straightforward (and more realistic).
For instance, for piecewise smooth $\partial\domDelta$ one could exploit that the measurement $E$ hits the set of nondifferentiability with probability zero.
The map $g$ in \cref{LemmaForwardDensitySplitting} and thus $\RNderivative{\Aq{\rho}}{\nu}$ in \eqref{eqn:basicFunctional} then will be continuous except on this set
so that our reconstruction functional stays well-defined almost surely.
In the more general case one would have to replace the reference measure $\nu$ from \cref{sec:likelihood} by something more appropriate,
for instance a projection of the Hausdorff measure on the four-dimensional torus $S^2\times S^2$ onto $\boundDomDelta$.
\end{remark}

\subsection{Forward operator}\label{sec:forwardOperator}

We start by introducing the linear forward operator $A\colon\measp(\dom)\to\measp(\boundDomDelta)$ describing the measurement process, i.e. the transformation of radioactive positron decays into photon detections in the detector domain. The forward operator is a weighted sum
\begin{equation*}
A=p^\att A^\att+p^\sct A^\sct+p^\dt A^\dt,
\end{equation*}
where the superscripts stand for the following three possibilities that a photon pair can undergo:
\begin{enumerate}[a)]
	\item \textbf attenuation: The emitted photon pair is not detected, for instance due to absorption.
	\item \textbf scattering: At least one of the photons undergoes substantial scattering, significantly altering its direction.
	\item \textbf detection: The photons undergo at most minor scattering
	before being registered by a pair of detectors.
\end{enumerate}
The parameters $p^\sct,p^\att,p^\dt=1-p^\sct-p^\att\in[0,1]$ denote the probabilities for scattering, attenuation, and scatterless detection, respectively.
For simplicity we assume those probabilities to be spatiotemporally constant,
but they could be replaced by functions of space and time without substantially modifying the approach.
During the remainder of the section we detail the three operators $A^\att$, $A^\sct$, and $A^\dt$.
The forward operator describing attenuation simply discards all intensity,
\begin{equation*}
A^\att\colon\measp(\dom)\to\measp(\boundDomDelta),\quad\rho_t\mapsto0.
\end{equation*}
Concerning scattering, we assume for simplicity that the photon rays are randomly redirected such that the probability of arriving at a point $(a,b)\in\boundDomDelta$ is homogeneous,
\begin{equation*}
A^\sct\colon \measp(\dom) \to \measp(\boundDomDelta),\quad
\rho_t \mapsto \frac{\rho_t(\dom)}{\hd^2(\partial\domDelta)^2}
\cdot
(\hd^2\otimes\hd^2)\restr(\boundDomDelta)
\end{equation*}
(our results could in principle also be extended to spatially inhomogeneous scattering with a correspondingly more elaborate operator $A^\sct$,
as long as the probability density of detecting a scattered photon pair is uniformly bounded away from zero,
that is, $\RNderivative{A^\sct\rho_t}{(\hd^2\otimes\hd^2)\restr(\boundDomDelta)}>\epsilon\rho_t(\dom)$ for some $\epsilon>0$ and any $\rho_t\in\measp(\dom)$).
The forward operator of scatterless detection will finally be modelled as a composition of linear operators
\begin{equation*}
A^\dt=B_{\mathrm{detectors}} B_{\mathrm{lines}} B_{\mathrm{pr}}.
\end{equation*}
The operator $B_{\mathrm{pr}}$ models the so-called \emph{positron range}:
The radioactive decay of an atom does actually not directly lead to photon emission but only produces a positron.
This positron may travel some small distance (depending on the material) before it annihilates with an electron,
which then results in the emission of the photon pair.
Let $G:B_{\delta/2}(0)\to[0,\infty)$ denote the (smooth, compactly supported) probability density of the annihilation location of a positron emitted at the origin,
where $B_r(x)$ denotes the open ball of radius $r$ centered at $x$.
Then, abbreviating $\domDeltaHalf=\dom+B_{\delta/2}(0)$, the convolution operator
\begin{equation*}
B_{\mathrm{pr}}\colon\measp(\dom)\to\measp(\domDeltaHalf),\;
\rho_t\mapsto G*\rho_t
\end{equation*}
transforms the intensity of radioactive decays into the intensity of photon emissions. We assumed that annihilation happens within $\domDeltaHalf$, i.e. within the detector's field of view, in order to get a well posed forward operator.
$G$ can also be used to approximately model minor scattering and small deviations of the photon-photon emission angle from $\pi$. In these cases the photon emission position is still close to the imaginary straight line between the two detector positions. This deviation could be modelled statistically and incorporated into the kernel $G$.
In the following we will therefore assume that the emitted photons travel exactly on a straight line.
Note that for simplicity of presentation we picked a spatially homogeneous kernel,
however, without changing the overall approach and the results one can make the kernel depend on position
(since in reality it depends on the surrounding material and the location of the detectors)
as long as one can find some $r>0$ such that the kernel stays uniformly bounded away from zero on the ball of radius $r$.
The next operator $B_{\mathrm{lines}}$ transforms the intensity of photon emission into an intensity on the space $\left(\domDeltaHalf\right)\times S^2$ of position-direction pairs, where $S^2=\{x\in\R^3 \ | \ \norm{x}=1\}$ denotes the sphere in $\R^3$.
A point $(x,v)\in\left(\domDeltaHalf\right)\times S^2$ stands for a photon pair emitted at $x$ along direction $v$. By choosing $v\in S^2$ we induce a symmetry in or forward model since $v$ and $-v$ correspond to the same photon pair.
On $S^2$ there is a natural normalized and uniform volume measure $\vol_{S^2}$, and since the photon directions after annihilation are distributed uniformly, the corresponding operator reads
\begin{equation*}
B_{\mathrm{lines}}\colon\measp(\left(\domDeltaHalf\right))\to\measp(\left(\domDeltaHalf\right)\times S^2),\;
\rho_t\mapsto\rho_t\otimes\vol_{S^2}\,.
\end{equation*}
Finally, each unscattered photon pair $(x,v)\in\left(\domDeltaHalf\right)\times S^2$ will be detected at the positions $R(x,v)$ with
\begin{equation*}
R\colon\left(\domDeltaHalf\right)\times S^2\to\boundDomDelta,\;
R(x,v)=\partial\domDelta\cap(x+\R v)\,,
\end{equation*}
where for simplicity we identify two-element subsets of $\partial\domDelta$ with a point in $\boundDomDelta$ in the following way: The point $(a,b)=R(x,v)\in\boundDomDelta$ is chosen such that $(b-a) / \absNorm{b-a} = v$. 
The operator transforming intensities in the space of photon pairs to intensities in the space of detector pairs thus is the pushforward
\begin{equation*}
B_{\mathrm{detectors}}\colon\measp(\left(\domDeltaHalf\right)\times S^2)\to\measp(\boundDomDelta),\;
\rho_t\mapsto\!\pushforward{R}{\rho_t}.
\end{equation*}

Later, we will work in a dynamic setting and consider measures that evolve in time. The forward operators on spacetime or path measures will be denoted by the same symbols:
It will be clear from the context whether the operators act on $\measp(\dom)$ or on $\measp(\spacetime)$ or on $\measp(C([0,T];\dom))$,
where the relation between them is
\begin{equation*}
\tilde A\theta=\tilde A\rho=\d t\otimes\tilde A\rho_t
\qquad\text{for }\tilde A=A,A^\dt,A^\sct,\text{ or }A^\att.
\end{equation*}
Note, that one could have modelled the directions of photon emissions using the Grassmannian manifold $G^{1,3}$ (one dimensional subspaces in $\R^3$) instead of $S^2$. Since a pair of photons is emitted in opposite directions, the line of emission has no natural sign and $G^{1,3}$ would be a more realistic model on first sight. Because we identified two-element subsets of $\partial\domDelta$ with a point in $\boundDomDelta$, using $G^{1,3}$ instead of $S^2$ would introduce an asymmetry in the detection part of the forward operator, which is why we decided to work with the double covering $S^2$. In the end, interpreting measurements adequately, both formulations are (in some sense) equivalent.
\begin{remark}[Extension of forward operator to $\R^3$]\label{rmk:GeneralizationOfForwardOperatorToR3}
	We defined the forward operator for measures on \dom. This can easily be extended to measures on $\R^3$ by applying the forward operator to the measure restricted to D. This is reasonable as real scanners for example discard events that have happened too closely to the detectors (which would happen for measures having support outside of \dom).
\end{remark}

\subsection{Bayesian ansatz and variables}\label{sec:BayesianAnsatz}
Since the considered underlying spaces are infinite dimensional, the main difficulty in the Bayesian modelling in our case is finding a suitable base measure with respect to which we can write down a density of the prior distribution. 
Sometimes this problem is circumvented by discretization.
For instance we may partition the domain $\dom$ into $\numRegionsA$ regions $\dom^1,\ldots,\dom^\numRegionsA$ of equal size,
thereby discretizing the measure $\rho_0\in\measp(\dom)$ as the vector $r^\numRegionsA=(\rho_0(\dom^1),\ldots,\rho_0(\dom^\numRegionsA))\in[0,\infty)^\numRegionsA$.
On $[0,\infty)^\numRegionsA$ one then considers the Lebesgue measure as the canonical base measure
and models the prior distribution of $\bm r^\numRegionsA$ by a density that is independent of the spatial mass distribution,
for instance $\density{\bm r^\numRegionsA}(r^\numRegionsA)=\prod_{i=1}^\numRegionsA\exp(-r_i^\numRegionsA)=\exp(-\|\rho_0\|)$
(our reasoning will be independent of the specific choice).
The density thus obtained turns out to be independent of the chosen discretization:
If we subdivide each region $\dom^i$ into $\numRegionsB$ smaller regions to obtain $\bar \numRegionsA=\numRegionsA \numRegionsB$ regions in total and thus a finer discretization $r^{\bar \numRegionsA}\in[0,\infty)^{\bar \numRegionsA}$ of $\rho_0$,
then following the same ansatz we again pick the density $\density{\bm r^{\bar \numRegionsA}}(r^{\bar \numRegionsA})=\prod_{i=1}^{\bar\numRegionsA}\exp(-r_i^{\bar \numRegionsA})=\exp(-\|\rho_0\|)$.
Thus one may be tempted to pick $\density{\bm\rho_0}(\rho_0)=\exp(-\|\rho_0\|)$ as the prior density for $\rho_0$ needed for the Bayesian reconstruction functional.
However, in doing so one overlooks the fact that the chosen prior distributions at the different discretization levels are actually incompatible with each other!
Indeed, if $r^{\bar \numRegionsA}$ is distributed according to $\density{\bm r^{\bar \numRegionsA}}(r^{\bar \numRegionsA})(\mathcal L\restr[0,\infty))^{\bar \numRegionsA}$,
then the probability to have at most mass $r_1^\numRegionsA$ in the coarser subdomain $\dom^1$ reads
\begin{multline*}
P(\bm r_1^{\bar \numRegionsA}+\ldots+\bm r_\numRegionsB^{\bar \numRegionsA}\leq r_1^\numRegionsA)
=\int_0^{r_1^\numRegionsA}\int_0^{{r_1^\numRegionsA}-r_\numRegionsB^{\bar \numRegionsA}}\ldots\int_0^{{r_1^\numRegionsA}-r_\numRegionsB^{\bar \numRegionsA}-\ldots-r_2^{\bar \numRegionsA}}\prod_{i=1}^\numRegionsB\exp(-r_i^{\bar \numRegionsA})\,\d r_1^{\bar \numRegionsA}\ldots\d r_\numRegionsB^{\bar \numRegionsA}\\
=1-\exp(-{r_1^\numRegionsA})\sum_{i=0}^{\numRegionsB-1}\frac{(r_1^\numRegionsA)^i}{i!},
\end{multline*}
as can readily be verified via induction in $\numRegionsB$.
Thus, $\bm r_1^\numRegionsA$ is distributed according to the density $\partial P(\bm r_1^{\bar \numRegionsA}+\ldots+\bm r_\numRegionsB^{\bar \numRegionsA}\leq r_1^\numRegionsA)/\partial r_1^\numRegionsA=\exp(-r_1^\numRegionsA)\frac{(r_1^\numRegionsA)^{\numRegionsB-1}}{(\numRegionsB-1)!}$
with respect to the Lebesgue measure, and not according to the density $\exp(-r_1^\numRegionsA)$ that we chose on the coarser discretization level!\\
	
Hence, we follow a different path and instead assume that the radioactive material is lumped into small particles that start in $\dom$ and travel around in $\R^3$ (allowing the particles to move in $\R^3$ instead of just $\dom$ simplifies the modelling) over the time interval $[0,T]$. \\
This assumption adequately describes many situations, for instance, if the radioactive material is carried by travelling leukocytes or other cells (which then represent the above particles).
Even if the radioactive material actually behaves like a diffuse quantity,
then on a mesoscale an infinitesimal volume element still contains many radioactive atoms and can thus be thought of as an imaginary radioactive particle.
For simplicity we consider the situation in which the radionuclide half-life $\halflife$ is much longer than the measurement time $T$
so that the radioactive particles may be assumed to radiate at constant rate.
We will later in \cref{rem:radiodecay} comment on how the approach has to be modified if $T$ becomes comparable to $\halflife$.

As a consequence, the sought quantity will be a nonnegative measure $\theta\in\measp(C_D([0,T];\R^3))$ on the space $C_D([0,T];\R^3)$ of continuous curves in $\R^3$ starting in $\dom$ with its Borel $\sigma$-algebra:
Each path corresponds to a particle trajectory, and the measure indicates which amount of particles follows a given set of trajectories.
The temporally changing material distribution is then given by $\rho=\d t\otimes\rho_t$ with $\rho_t=\Mv_t\theta$ for the mass moving operator
\begin{equation}\label{eq:DefMassMovingOperator}
\Mv_t:\meas(C_D([0,T];\R^3))\to\meas(\R^3),
\
\Mv_t\theta=\pushforward{\eval_t}\theta
\quad\text{with }
\eval_t:C_D([0,T];\R^3)\to\R^3,
\
\eval_t\gamma=\gamma(t).
\end{equation}

The radioactive decay and associated photon emission then happens according to a Poisson point process with intensity
\begin{equation*}
\d t\otimes\lambda_t=\d t\otimes\frac{\ln2}{\halflife}\rho_t,
\end{equation*}
Equivalently, the number of decays within a time interval $[t_1,t_2)$ and subset $B\subset\R^3$ is Poisson-distributed with mean $\int_{t_1}^{t_2}\lambda_t(B)\,\d t$.
To simplify the notation we will in the following simply neglect the factor $\ln 2$ as this does not change any of the calculations (alternatively one could consider a rescaled mass distribution $\theta$ and $\rho$). Similarly, the value of the half-life does not influence most of the calculations so that we may without loss of generality consider $\halflife=1$ unless otherwise stated and we will use $\rho_t$ instead of $\lambda_t$.

As already explained in the introduction, each radioactive decay produces a photon pair which eventually is detected on $\boundDomDelta$.
As a consequence, the measurement $E$ will also be a realization of a Poisson point process, this time on $[0,T]\times\boundDomDelta$,
with intensity $\d t\otimes A\rho^\dagger_t$ ($\dt t\otimes\rho^\dagger_t$ is the ground truth material density to be reconstructed),
\begin{equation*}
\boldsymbol E=\Poi{\d t\otimes A\rho^\dagger_t},
\end{equation*}
where the linear forward operator $A\colon\measp(\R)\to\measp(\boundDomDelta)$ describes the detection process in the PET scanner (see \cref{rmk:GeneralizationOfForwardOperatorToR3} for the generalization of $A$ to $\meas_+(\R^3)$). Since Poisson point processes are proper (up to equality in distribution, see \cite[Cor. 3.7]{bookPPP_last_penrose}), we may restrict our considerations to such processes without loss of generality. A measurement thus is almost surely a Radon measure of the form $E=\sum_{k=1}^K\delta_{(t_k,a_k,b_k)}$.

Our task is to reconstruct $\theta$ given a measurement $E$.
Taking a Bayesian approach, $(E,\theta)$ is viewed as realization of a random variable $(\bm E,\bm\theta)$
with some joint probability distribution of the form $\density{\bm E,\bm\theta}(E,\theta)\baseMeasMeas\otimes \baseMeasPart$
where $\density{\bm E,\bm\theta}$ is a density and $\baseMeasMeas$ and $\baseMeasPart$ are base or reference measures (not necessarily probability measures)
on the space of measurements and particle configurations, respectively.
By integrating with respect to $\theta$ or $E$ one obtains the probability distributions $\density{\bm E}\baseMeasMeas$ of measurements and $\density{\bm\theta}\baseMeasPart$ of particle configurations with
\begin{equation*}
\density{\bm E}(E)=\int \density{\bm E,\bm\theta}(E,\theta)\,\d \baseMeasPart(\theta)
\qquad\text{and}\qquad
\density{\bm\theta}(\theta)=\int \density{\bm E,\bm\theta}(E,\theta)\,\d \baseMeasMeas(E)
\end{equation*}
as well as the conditional probability distributions $\density{\bm E|\bm\theta}(E|\theta)\baseMeasMeas$ of $\bm E$ given $\theta$ and $\density{\bm\theta|\bm E}(\theta|E)\baseMeasPart$ of $\bm \theta$ given $E$ with
\begin{equation*}
\density{\bm E|\bm\theta}(E|\theta)=\density{\bm E,\bm\theta}(E,\theta)/\density{\bm\theta}(\theta)
\qquad\text{and}\qquad
\density{\bm\theta|\bm E}(\theta|E)=\density{\bm E,\bm\theta}(E,\theta)/\density{\bm E}(E).
\end{equation*}
The function $\density{\bm E| \bm \theta}(E|\theta)$ is called the \emph{likelihood};
given a particular measurement $E$, a particle configuration $\theta$ will then be considered the more likely the higher the associated likelihood value $\density{\bm E|\bm\theta}(E|\theta)$ is
(the maximizing $\rho$ is known as maximum likelihood estimate).
The function $\density{\bm\theta|\bm E}(\theta|E)$ is the so-called \emph{posterior distribution} and can by Bayes' rule be expressed as
\begin{equation*}
\density{\bm\theta|\bm E}(\theta|E)=\density{\bm E|\bm\theta}(E|\theta)\density{\bm\theta}(\theta)/\density{\bm E}(E).
\end{equation*}
Its maximizer is known as maximum a posteriori estimate and will be our reconstruction.
Equivalently one expresses the reconstruction as minimizer of the negative log-posterior $-\log\density{\bm\theta|\bm E}(\theta|E)$, which is the functional we aim to derive.

Note that above we simply assumed the distribution of $(\bm E,\bm \theta)$ to have a density with respect to $\baseMeasMeas\otimes \baseMeasPart$.
In fact this structure is implied by the existence of the (conditional) probability distributions $\density{\bm E|\bm\theta}(E|\theta)\baseMeasMeas$ and $\density{\bm\theta}\baseMeasPart$,
so in the remainder of this section we will provide appropriate base measures $\baseMeasMeas$ and $\baseMeasPart$,
model the prior distribution $\density{\bm\theta}\baseMeasPart$, and derive the likelihood $\density{\bm E|\bm\theta}(E|\theta)$.
Now it turns out that while we can readily provide a reasonable prior distribution,
the latter is difficult to express via a density function $\density{\bm\theta}$
since a canonical base measure $\baseMeasPart$ on the infinite-dimensional space $\measp(C_D([0,T];\R^3))$ is lacking. We will therefore follow a trick (which is for instance also used when defining the Wiener measure):
Instead of aiming for a full reconstruction of $\theta$ we just reconstruct
\begin{multline*}
\theta_{t_0,\ldots,t_K}=\pushforward{\eval_{t_0,\ldots,t_K}}\theta\in\measp(\dom\times(\R^3)^K)\\
\quad\text{with }
\eval_{t_0,\ldots,t_K}:C_D([0,T];\R^3)\to\dom\times(\R^3)^K,
\
\eval_{t_0,\ldots,t_K}\gamma=(\gamma(t_0),\ldots,\gamma(t_K))
\end{multline*}
for a fixed chosen number of time points $t_0,\ldots,t_K\in[0,T]$ (containing also $t_0=0$ as well as the time points where a photon pair was recorded).
The projection of the prior distribution for $\theta$ onto the distribution of $\theta_{t_0,\ldots,t_K}$ can more easily be expressed as a density times a base measure.
The final reconstruction functional will then in fact turn out to be independent of the choice of included time points.
Furthermore, we will see that the measurement $E$ (whose photon detection time points are included in $\{t_0,\ldots,t_K\}$) satisfies
\begin{equation*}
\density{\bm E|\bm\theta}(E|\theta)
=\density{\bm E|\bm{\theta}_{t_0,\ldots,t_K}}(E|\theta_{t_0,\ldots,t_K})
\end{equation*}
(the latter being the conditional density of measurements given the partial information $\theta_{t_0,\ldots,t_K}$ of the particle configuration).
Thus, for our Bayesian reconstruction functional we slightly modify our above plan for the remainder of the section:
We will provide base measures $\baseMeasMeas$ and $\baseMeasPartCondTimePoints{t_0,\ldots,t_K}$ (base measure for the partial information $\theta_{t_0,\ldots,t_K}$ of the particle configuration),
model the prior distribution $\density{\bm\theta_{t_0,\ldots,t_K}}\baseMeasPartCondTimePoints{t_0,\ldots,t_K}$, and derive the likelihood $\density{\bm E|\bm\theta}(E|\theta)$.

\subsection{Likelihood function}\label{sec:likelihood}
As explained before, given a radioactive material density $\rho=\d t\otimes\rho_t=\d t\otimes\Mv_t\theta\in\measp([0,T]\times\R^3)$,
the resulting measurements can be described (up to the factor $\ln2/\halflife$ which we agreed to ignore) by the Poisson point process
\begin{equation*}
\bm E=\Poi{A\theta}=\Poi{\d t\otimes A\rho_t}
\end{equation*}
on $[0,T]\times\boundDomDelta$.
We also discussed that for a Bayesian model
we need to express the probability distribution of observations $E$ of $\bm E$ as a density $\density{\bm E|\bm \theta}(E|\theta)$, the likelihood, with respect to some base measure $\baseMeasMeas$.
At first glance it may seem that the choice of the base measure $\baseMeasMeas$ influences the final Bayesian reconstruction functional $\hat J^E$ and thus the reconstruction,
however, in the end this will actually not be the case due to properties of Poisson processes
($\hat J^E$ will be independent of $\baseMeasMeas$).
Hence we may choose any $\baseMeasMeas$ such that the distribution of $\bm E=\Poi{A\theta}=\Poi{A\rho}$ has a density with respect to $\baseMeasMeas$.
To this end, by \cite[Thm.\,3.1.1]{ReissPP} we may for instance choose $\baseMeasMeas$ to be the distribution of $\Poi{\nu}$
for any finite measure $\nu\in\measp([0,T]\times\boundDomDelta)$ with respect to which $A\theta=A\rho=\d t\otimes A\rho_t$ is absolutely continuous,
and the corresponding density reads
\begin{equation*}
\density{\bm E|\bm\theta}(E|\theta)=\left(\prod_{(t,a,b)\in E}\RNderivative{A\theta}{\nu}(t,a,b)\right)\exp\left( ||\nu||-||A\theta|| \right).
\end{equation*}
Later, in \cref{LemmmaBoundednessForwardOp}, we will show that $\d t\otimes A\rho_t$ is absolutely continuous with respect to
\begin{equation*}
\nu=\d t\otimes(\hd^2\restr\partial\domDelta)\otimes(\hd^2\restr\partial\domDelta)
\end{equation*}
for any measure $\rho\in\measp([0,T]\times\R^3)$ of the form $\rho=\d t\otimes\rho_t$, so we fix this choice of $\nu$ from now on.

Next assume that the measurement $E$ detects photon pairs at a subset of the times $t_0,\ldots,t_K$.
The conditional probability density of this particular realization $E$ with respect to $\baseMeasMeas$, given $\theta_{t_0,\ldots,t_K}$, then reads
\begin{equation}\label{eq:LikelihoodTimeConstant}
\density{\bm E|\bm\theta_{t_0,\ldots,t_K}}(E|\theta_{t_0,\ldots,t_K})
=\int\density{\bm E|\bm\theta}(E|\theta)\,\d \probMeasPartCondTimePoints{\bm \theta}{\bm\theta_{t_0,\ldots,t_K}}(\theta|\theta_{t_0,\ldots,t_K}),
\end{equation}
where $\probMeasPartCondTimePoints{\bm \theta}{\bm\theta_{t_0,\ldots,t_K}}(\theta|\theta_{t_0,\ldots,t_K})$ denotes the conditional probability distribution of $\theta$
given its (continuous) projection $\theta_{t_0,\ldots,t_K}=\pushforward{\eval_{t_0,\ldots,t_K}}\theta$.
Now note that our above expression for $\density{\bm E|\bm\theta}(E|\theta)$ actually only depends on $\RNderivative{A\theta}{\nu}$ at the measurement times.
Indeed, assuming constant mass in time inside $\dom$ (this is consistent with the final reconstruction formula \cref{eqn:finalFunctional} where the continuity equation constraint implies constant mass in time) we have
\begin{equation*}
\density{\bm E|\bm\theta}(E|\theta)=\left(\prod_{(t,a,b)\in E}\RNderivative{A\rho_t}{(\hd^2\restr\partial\domDelta)\otimes(\hd^2\restr\partial\domDelta)}(a,b)\right)\exp\left( ||\nu||-T||A\rho_{t_0}|| \right).
\end{equation*}
Therefore the integrand in \cref{eq:LikelihoodTimeConstant} is constant so that
\begin{equation}\label{eqn:likelihood}
\density{\bm E|\bm\theta_{t_0,\ldots,t_K}}(E|\theta_{t_0,\ldots,t_K})
=\density{\bm E|\bm\theta}(E|\theta)\int\,\d \probMeasPartCondTimePoints{\bm \theta}{\bm\theta_{t_0,\ldots,t_K}}(\theta|\theta_{t_0,\ldots,t_K})
=\density{\bm E|\bm\theta}(E|\theta).
\end{equation}

\subsection{Prior distribution}
Next we need to model the prior distribution of the random variable $\bm\theta$
and then express it or, as discussed in \cref{sec:BayesianAnsatz}, rather express the induced prior distribution of $\bm\theta_{t_0,\ldots,t_K}$
as a density $\density{\bm\theta_{t_0,\ldots,t_K}}(\theta_{t_0,\ldots,t_K})$ with respect to some base measure $\baseMeasPartCondTimePoints{t_0,\ldots,t_K}$.
Our model will be based on the following considerations:
We aim for a generic prior distribution that encompasses all kinds of spatiotemporal particle motions
and therefore abstain from using a specific physical or biological model.
Hence, without further information on the particle motion we should assume it to be more or less random,
so as prior distribution for each single particle trajectory we will assume the distribution of Brownian motions starting in $\dom$.
In addition we will have to specify the initial distribution of particles, which we will try to do as uniformly as possible.

Consider first the initial spatial particle configuration $\bm\rho_0$.
Intuitively we would like to put no prior information in and to let all configurations be equally likely
(except for maybe a decreasing likeliness with increasing total particle mass).
This intuition is a little deceptive, though:
In order to say that two configurations are equally likely we already have to implicitly assume the existence of some base measure on the space $\measp(\dom)$ of particle configurations
(so that we can evaluate the probability density at both configurations to compare their likeliness).
However, there is no canonical base measure on an infinite-dimensional space. An appropriate approach to achieve a spatially uniform prior probability distribution for the initial particle configuration $\bm \rho_0$ is achieved
by taking into account that the radioactivity is actually quantized by the radioactivity $\quant$ of a single atom.
One can thus consider the initial material distribution $\rho_0$ as $\quant$ times a realization of a Poisson point process on $\dom$ with uniform intensity
(to make the total amount of radioactive material independent of the domain and the employed radionuclide this intensity should be $\frac1{\quant\mathcal L^3(\dom)}\mathcal L^3\restr\dom$).
In fact, this distribution serves equally well as prior distribution and as the base measure with respect to which the prior distribution is expressed as a density:
It is canonical in the sense that the location of each particle is uniformly distributed on $\dom$ (the amount of particles is Poisson distributed, see \cite[Section 1.2]{ReissPP}). Moreover, it was shown in \cite{defVolumeMeasure} that (mixed) Poisson measures $\mu$ are exactly those measures making gradient and divergence operators dual operators on $L^2(\Gamma_X, \mu)$, where $\Gamma_X$ is the so-called configuration space over $X$ which is the space of all locally finite point measures. This makes Poisson measures a natural choice for a volume measure on the configuration space as they resemble properties of the Lebesgue measure on the Euclidean space.\\
Starting from the initial material distribution $\bm \rho_0$ each particle now moves independently according to a Brownian motion. According to the discussion above we only consider the marginals of the Brownian motions at the time points $t_0,\ldots,t_K$, where the positions at $t_0$ are determined via $\bm\rho_0$. Fixing the number of particles to be $J$ for the moment we view any realization of $\bm\theta_{t_0,\ldots,t_K}$ as a quantized sum $\theta_{t_0,\ldots,t_K}=\quant\sum_{i=1}^J\delta_{(x_0^i,\ldots,x_K^i)}$ and for each $i$ the positions $x_1^i,\ldots,x_K^i$ are drawn from a $K$-dimensional joint normal distribution on $(\R^3)^K$, whereas the initial locations $x_0^i$ are drawn from a uniform distribution \cite[Lemma 1.2.1]{ReissPP}
\begin{equation*}
\frac{\mathcal L^3\restr\dom}{\quant\mathcal L^3(\dom)\frac{\mathcal L^3(\dom)}{\quant\mathcal L^3(\dom)}} = \frac1{\mathcal L^3(\dom)}\mathcal{L}^3\restr\dom
\end{equation*}
 This leads to the density 
\begin{equation*}
\density{(\bm x^i_0,\ldots,\bm x^i_K)_{i=1}^J}((x^i_0,\ldots,x^i_K)_{i=1}^J)=\prod_{i=1}^\numParticlesBayes\prod_{k=1}^K\exp\left(-\tilde\beta\frac{|x_k^i-x_{k-1}^i|^2}{t_k-t_{k-1}}\right)\tilde h(J)
\end{equation*}
for some diffusion coefficient $\frac{1}{2\tilde\beta}>0$ and w.r.t. the base measure $\baseMeasPartTimePoints{(\bm x_0^i,\ldots,\bm x_K^i)_{i=1}^J}^J$ (being a scaled multidimensional Lebesgue measure)
\begin{align*}
\baseMeasPartTimePoints{(\bm x_0^i,\ldots,\bm x_K^i)_{i=1}^J}^J(( \d x_0^i,\ldots,\d x_K^i)_{i=1}^\numParticlesBayes)=h(J)\left(\frac{1}{\mathcal L^3(\dom)}\right)^J\bigotimes_{i=0}^J\left((\mathcal{L}^3\restr\dom)(\d x_0^i)\bigotimes_{k=1}^K\mathcal{L}^3(\d x_k^i)\right)
\end{align*}
for the functions $\tilde h(J)=1$ and
\begin{align*}
h\colon\N\to\R,\quad h(\numParticlesBayes)=\prod_{k=1}^K\left(\frac{\tilde\beta}{\pi(t_k-t_{k-1})}\right)^{3\numParticlesBayes/2}.
\end{align*}
At first sight, this choice seems rather arbitrary. However, it is a natural choice since otherwise the maximizing atom number $\numParticlesBayes$ of the prior density would depend on the chosen times $t_0,\ldots,t_K$.
In other words, only by this choice of $h$ neither the prior distribution nor its density induce a bias on $\numParticlesBayes$ that depends on the measurement times.\\
Finally, we include the randomness of the number of particles $\numParticlesBayes$ and arrive at the density
\begin{align*}
\density{\bm\theta_{t_0,\ldots,t_K}}(\theta_{t_0,\ldots,t_K})=\density{(\bm x_0^i,\ldots,\bm x_K^i)_{i=1}^J}((x_0^i,\ldots,x_K^i)_{i=1}^J)\quad\text{for }\theta_{t_0,\ldots,t_K}=\quant\sum_{i=1}^J\delta_{(x_0^i,\ldots,x_K^i)}
\end{align*}
w.r.t. the measure
\begin{align*}
\baseMeasPartCondTimePoints{t_0,\ldots,t_K}(\d\theta_{t_0,\ldots,t_K}) = \baseMeasPartTimePoints{(\bm x_0^i,\ldots,\bm x_K^i)_{i=1}^\numParticlesBayes}^\numParticlesBayes(( \d x_0^i,\ldots,\d x_K^i)_{i=1}^\numParticlesBayes)\otimes\mathrm{Poi}_{\mathcal{L}^3(\dom)/\quant}(\d\numParticlesBayes),
\end{align*}
where $\mathrm{Poi}_{\mathcal{L}^3(\dom)/\quant}$ is the Poisson distribution with parameter $\mathcal{L}^3(\dom)/\quant$. Since we chose both the prior distribution and base measure for the initial particle configuration to be the same, we do not get additional factors in the final density. Note, that for every fixed number of particles $J$
\begin{align*}
\density{(\bm x_0^i,\ldots,\bm x_K^i)_{i=1}^J}((x_0^i,\ldots,x_K^i)_{i=1}^J)\baseMeasPartTimePoints{(\bm x_0^i,\ldots,\bm x_K^i)_{i=1}^\numParticlesBayes}^\numParticlesBayes(( \d x_0^i,\ldots,\d x_K^i)_{i=1}^\numParticlesBayes)
\end{align*}
has unit mass by our choices of $h$ and $\tilde h$.

Consequently, if multiple radioactive atoms happen to travel together in $\numParticlesBayes$ particles of masses $m_1,\ldots,m_\numParticlesBayes$, then the density turns into
\begin{equation}\label{eqn:priorDensity}
\density{\bm\theta_{t_0,\ldots,t_K}}(\theta_{t_0,\ldots,t_K})
=\prod_{i=1}^\numParticlesBayes\prod_{k=1}^K\exp\left(-m_i\beta\frac{|x_k^i-x_{k-1}^i|^2}{t_k-t_{k-1}}\right)
\quad\text{if }
\theta_{t_0,\ldots,t_K}=\sum_{i=1}^\numParticlesBayes m_i\delta_{(x_0^i,\ldots,x_K^i)},
\end{equation}
where we abbreviated $\beta=\tilde\beta/\quant$.

We derived this prior density assuming that only quantized measures $\theta_{t_0,\ldots,t_K}$ can occur.
However, an optimization over quantized measures is difficult, so we simply extend the above density to all discrete nonnegative measures $\theta_{t_0,\ldots,t_K}$ (also nonquantized ones).
On the level of the optimization problem introduced in the next paragraph this corresponds to performing a standard convex relaxation.

\begin{remark}[Alternative to quantization]
Instead of working with quantized measures so that a Poisson point process can serve as the base measure,
one could also consider base measures that allow distributions of particles with random locations as well as random masses.
To this end one would employ so-called marked Poisson point processes in which the mark assigns each point a random mass.
However, in that case one would have to propose a model of how the diffusion constant of the Brownian motion should depend on that mass,
and it woud be more involved to take care that the choice of times $t_0,\ldots,t_K$ does not influence the prior.
\end{remark}

\begin{remark}[Discrete particle paths]
We only defined our density $\density{\bm\theta_{t_0,\ldots,t_K}}$ for discrete measures $\theta_{t_0,\ldots,t_K}$
(a natural relaxation to arbitrary nonnegative measures will be performed in the next paragraph).
Later in \cref{thm:minimizerStructure} we will show that indeed there exists a reconstruction that consists of finitely many particles.
At first sight one might be worried that those reconstructions are all very special
in that many atoms are lumped together in larger particles and travel together
(a situation that has probability zero with respect to our above introduced base measure $P_{\bm\theta_{t_0,\ldots,t_K}}$).
However, actually this is the expected behaviour of MAP estimates:
The most likely configuration often shows much stronger regularity than a typical configuration.
\end{remark}

\subsection{Posterior distribution and Bayesian functional}\label{sec:posterior}
As usual let the measurement times in the measurement $E$ be a subset of $\{t_0,\ldots,t_K\}$,
and abbreviate $\rho_{t_k}=\pushforward{\pi_k}{\theta_{t_0,\ldots,t_K}}$ for the projection $\pi_k:(x_1,\ldots,x_K)\mapsto x_k$.
By Bayes' rule we have
\begin{equation*}
\density{\bm\theta_{t_0,\ldots,t_K}|\bm E}(\theta_{t_0,\ldots,t_K}|E)
=\density{\bm E|\bm\theta_{t_0,\ldots,t_K}}(E|\theta_{t_0,\ldots,t_K})\density{\bm\theta_{t_0,\ldots,t_K}}(\theta_{t_0,\ldots,t_K})/\density{\bm E}(E).
\end{equation*}
Taking the negative logarithm and inserting the expressions from \eqref{eqn:likelihood} and \eqref{eqn:priorDensity}
we arrive (up to an additive constant depending on $E$) at the following Bayesian reconstruction functional for $\theta_{t_0,\ldots,t_K}=\sum_{i=1}^\numParticlesBayes m_i\delta_{(x_1^i,\ldots,x_K^i)}$,
\begin{equation*}
\tilde J^E(\theta_{t_0,\ldots,t_K})
=T\|A\rho_{t_0}\|-\sum_{(t,a,b)\in E}\log\left(\RNderivative{A\rho_t}{(\hd^2\restr\partial\domDelta)\otimes(\hd^2\restr\partial\domDelta)}(a,b)\right)
+\beta\sum_{k=1}^K\sum_{i=1}^\numParticlesBayes m_i\frac{|x_k^i-x_{k-1}^i|^2}{t_k-t_{k-1}},
\end{equation*}
whose minimizer is the desired Bayesian reconstruction. Note, that the above functional is oblivious to stationary mass being placed outside of $\dom$ since neither of the three different components of the functional would be influenced by this. Hence, we can restrict the reconstruction to particle positions inside $\dom$.
By associating with $\theta_{t_0,\ldots,t_K}$ the measure $\theta=\sum_{i=1}^\numParticlesBayes m_i\delta_{\gamma_i}\in\measp(C([0,T];\dom))$
with $\gamma_i$ the piecewise linear interpolation of the points $x_0^i,\ldots,x_K^i$ at times $t_0,\ldots,t_K$, we obtain
\begin{multline*}
\tilde J^E(\theta_{t_0,\ldots,t_K})
=\min_{\substack{\theta\in\measp(C([0,T];\dom))\\\theta_{t_0,\ldots,t_K}=\pushforward{\eval_{t_0,\ldots,t_K}}\theta}}\bar J^E(\theta)\\
\text{for }
\bar J^E(\theta)
=\|A\theta\|-\sum_{(t,a,b)\in E}\log\left(\RNderivative{A\theta}{\nu}(t,a,b)\right)
+\beta\int\int_0^T|\dot\gamma(t)|^2\,\d t\d\theta(\gamma).
\end{multline*}
The functional $\tilde J^E$ obviously is independent of the chosen time points $t_0,\ldots,t_K$ (as long as they contain the measurement time points)
so that its minimizer may be viewed as the Bayesian reconstruction of the particle configuration $\theta\in\measp(C([0,T];\dom))$.
Finally, the equivalence of the Benamou--Brenier formulation to a generalized flow formulation of optimal transport \cite[\S\,4.2 \& 5.3]{Be21} implies
\begin{multline}\label{eqn:DefJhat}
\min_{\theta\in\measp(C([0,T];\dom))}\bar J^E(\theta)
=\min_{\rho\in\measp([0,T]\times\dom),\momentum\in\meas([0,T]\times\dom)^3}\hat J^E(\rho,\momentum)\\
\text{for }
\hat J^E(\rho,\momentum)=
\|A\rho\|-\sum_{(t,a,b)\in E}\log\left(\RNderivative{A\rho}{\nu}(t,a,b)\right)
+\beta\BBEnergy(\rho,\momentum),
\end{multline}
where the optimal $\theta$ and $(\rho,\momentum)$ are related by $\rho=\d t\otimes\Mv_t\theta$ and $\momentum=\d t\otimes\Mv_t(\theta\dot\gamma)$.
Thus, the optimal $\rho$ and $\momentum$ describe the temporally changing radioactive mass distribution and mass flux associated with the particle configuration $\theta$.
The functional $\hat J^E$ is that formulation of the negative log posterior that we aimed for and which we will consider in the remainder of the article (up to subtle changes due to a bias removal, see next section).
In essence, its Benamou--Brenier term implies that we assign a higher likeliness to a path $(\rho,\momentum)$ the less mass moves along that path.
This acts as a temporal or kinetic regularization of the particle trajectories.

\begin{remark}[Nonnegligible radioactive decay]\label{rem:radiodecay}
As mentioned previously, we assumed $T\ll\halflife$ so that the amount of radioactive material does essentially not change over time
and the continuity equation \eqref{eqn:continuity} is valid.
If in contrast $T$ becomes comparable to $\halflife$,
then the continuity equation \eqref{eqn:continuity} has to be complemented with an additional decay term,
\begin{equation*}
\partial_t\rho+\div\momentum=-\frac{\ln2}{\halflife}\rho
\end{equation*}
(which then is related to so-called unbalanced optimal transport \cite{ChPeScVi18}).
\end{remark}

\begin{remark}[Alternative derivation]
The same functional $\hat J^E$ could also have been obtained by assuming a different prior density $\density{\bm\theta_{t_0,\ldots,t_K}}(\theta_{t_0,\ldots,t_K})$:
We could have chosen
\begin{equation*}
\density{\bm\theta_{t_0,\ldots,t_K}}(\theta_{t_0,\ldots,t_K})
\sim\exp\left(-\beta\sum_{k=2}^K \frac{W_2^2(\rho_{t_{k-1}},\rho_{t_k})}{t_k-t_{k-1}}\right),
\end{equation*}
where again $\rho_{t_k}=\pushforward{\pi_k}{\theta_{t_0,\ldots,t_K}}$ and $W_2^2$ denotes the squared Wasserstein distance.
However, with the Wasserstein optimal transport distance there comes along an identification of particles between $\rho_{t_{k-1}}$ and $\rho_{t_k}$,
and this does not have to correspond to the actual identification encoded in $\theta_{t_0,\ldots,t_K}$.
Only after finding the most likely reconstruction by minimizing the Bayesian reconstruction functional
the optimized variables $\rho_{t_k}$ and $\theta_{t_0,\ldots,t_K}$ allow a consistent particle identification.
\end{remark}

\subsection{Discrete measurements}\label{sec:discrete}

In practice, a measurement is not taken in the continuous spacetime setting, but rather in a discretized fashion:
The photons are collected in detectors that have a nonvanishing spatial extent, and detected photon pairs are binned into time intervals of (short, but) positive length $\Delta T$.
The photon detectors partition $\partial\domDelta$ into $M$ measurable disjoint regions
\begin{equation*}
\Gamma_j\subset\partial\domDelta
\qquad\text{ with }\hd^2(\partial\Gamma_j)=0,\quad j=1,\ldots,M,
\end{equation*}
and the measurement time is partitioned into $N$ disjoint time bins
\begin{equation*}
\tau_i=[(i-1)\Delta T,i\Delta T),\quad i=1,\ldots,N=\tfrac T{\Delta T}.
\end{equation*}
A measurement $E$ then is a list of indices, indicating which detector pairs detected a photon pair in which time bins.
Equivalently, each detector $\Gamma_j$ is identified with a point $z_j\in\Gamma_j$ and each time bin $\tau_i$ with its centre $(i-\frac12)\Delta T$
so that as before the measurement consists of a list $E=(t_k,a_k,b_k)_{k=1,\ldots,K}\in(\R\times\partial \domDelta\times\partial \domDelta)^K$
with $a_k,b_k\in\{z_1,\ldots,z_M\}$ and $t_k\in\{\frac{\Delta T}2,\ldots,T-\frac{\Delta T}2\}$.

To account for the discrete measurement we modify the forward operator $A$ from \cref{sec:forwardOperator} by applying an additional discretization operator
\begin{equation*}
B_\discrete:\meas([0,T]\times\boundDomDelta)\to\meas([0,T]\times\boundDomDelta),\quad
B_\discrete\mu=\sum_{i=1}^N\sum_{j,k=1}^M\mu(\tau_i\times\Gamma_j\times\Gamma_k)\delta_{((i-\frac12)\Delta T,z_j,z_k)}.
\end{equation*}
To unify the notation for both the continuous and the discrete scenario,
in the latter we simply redefine $A=B_\discrete A_\cont$ for $A_\cont$ being the operator $A$ from \cref{sec:forwardOperator} and analogously for $A^\att$, $A^\sct$, and $A^\dt$.
Likewise we need to adapt the reference measure $\nu$ that we employed to calculate the likelihood in \cref{sec:likelihood}:
Since it has to dominate $A\rho$, in the discrete case we can simply redefine it as $\nu=\sum_{i=1}^N\sum_{j,k=1}^M\delta_{((i-\frac12)\Delta T,z_j,z_k)}$.
Apart from this modification of the operator $A$ and the measure $\nu$
the derivation of the Bayesian reconstruction functional $\hat J^E$ is exactly the same in the discrete setting as in the continuous one.

\begin{remark}[Vanishing positron range in discrete setting]\label{rem:discreteVsPositronRange}
Let us mention that in the discrete setting we can even deal with vanishing positron range
(for instance if the positron range is so small that one would like to neglect it in the modelling),
which corresponds to the convolution kernel $G$ from \cref{sec:forwardOperator} being the Dirac measure.
Indeed, in the discrete setting, due to the definition of the discretization operator $B_\discrete$ it is trivial that $\nu$ dominates $A\rho$ (whether with or without positron range)
so that our calculation of the likelihood in \cref{sec:likelihood} is valid.
This is not so in the continuous setting, in which one needs a nonvanishing positron range for $\nu$ to dominate $A\rho$ (compare \cref{LemmmaBoundednessForwardOp}).
\end{remark}

\section{Unbiasing for scattered events}
The MAP estimate is the minimizer of the Bayesian reconstruction functional $\hat J^E$ or equivalently $\tilde J^E$,
which in turn is the sum of the negative logarithms of the likelihood function $\density{\bm E|\bm\theta_{t_0,\ldots, t_K}}(E|\theta_{t_0,\ldots,t_K})$
and the prior density $\density{\bm\theta_{t_0,\ldots,t_K}}(\theta_{t_0,\ldots,t_K})$.
Unfortunately, this MAP estimate is strongly biased towards declaring every detected photon pair as being unscattered.
Essentially, this is due to the first summand, $-\log\density{\bm E|\bm\theta_{t_0,\ldots,t_K}}(E|\theta_{t_0,\ldots,t_K})$,
and more specifically due to the scatter operator $p^\sct A^\sct$ being very small compared to the operator $p^\dt A^\dt$ associated with unscattered photon pairs in the definition of $A$:
Indeed, $A^\sct$ distributes the intensity evenly over all detectors, while $A^\dt$ concentrates all the intensity on a few detectors and therefore produces much higher intensities there.
This is a well-known deficiency of the maximum likelihood (ML) estimator (the minimizer of just the first summand) and consequently also the MAP estimator,
another variant of which is the fact that, for a number of samples drawn from a fixed distribution, the empirical density
(which is the ML estimate of the distribution)
explains every sample by a Dirac mass at the sample position rather than a more evenly spread distribution.
In this section we provide a remedy for this deficiency.
To this end we first explore the bias in more detail in simplified 1D examples, then introduce and analyse our remedy, and finally showcase its effect on the simplified examples. As these example focus in the likelihood, no (kinetic) regularization for $\rho$ is considered.

\subsection{Bias of MAP estimate towards nonscattered events}\label{sec:bias}
Our simplified 1D examples consider the domain
\begin{equation*}
D=[0,1]
\end{equation*}
with some sought ground truth distribution $\rho^\dagger\in\measp(D)$ and a simplified forward operator $A:\measp(D)\to\measp(D)$.
For simplicity we will equip $D$ with periodic boundary conditions, that is, we will identify the location $0$ with $1$
so that convolution of periodic functions on $D$ is well-defined.

We start by looking at a continuous ground truth distribution and the simplest possible forward operator
\begin{equation*}
\rho^\dagger=\numParticles\mathcal{L}\restr D
\qquad\text{and}\qquad
A\rho=\rho
\end{equation*}
for some $\numParticles\in\N$.
Thus, the measurement $E$ is a realization of $\Poi{\rho^\dagger}$ and may for instance be given by $E=\sum_{i=1}^\numParticles\delta_{x_i}$ for some $x_1,\ldots,x_\numParticles\in D$.
The negative logarithm of the likelihood (taking $\nu=\mathcal{L}\restr D$) in this case reads
\begin{equation*}
\hat J^E(\rho)=\norm{\rho}-\int\log\left( \frac{\d\rho}{\d\mathcal{L}} \right)\d E=\norm\rho-\sum_{i=1}^\numParticles\log\left( \frac{\d\rho}{\d\mathcal{L}}(x_i)\right).
\end{equation*}
Now let $\RNderivative{\rho_\varepsilon}{\mathcal{L}}(x)=\sum_{i=1}^\numParticles\frac1\varepsilon\varphi_\varepsilon(\frac{x-x_i}\varepsilon)$ for a mollifier $\varphi:\R\to\R$ of compact support, unit mass and positive $\varphi(0)$.
Then $\rho_\varepsilon\weakstarto\rho^\dagger$ and $\hat J^E(\rho_\epsilon)\to-\infty$ as $\epsilon\to0$,
which shows the overwhelming tendency of the ML estimate to explain every single sample point by a Dirac measure,
even though the ground truth is a multiple of the Lebesgue measure.

Our next example is a little closer to the actual situation during a PET measurement, including positron range and scatter: Let $p^\sct\in[0,1]$ denote the probability of scatter and consider
\begin{equation*}
\rho^\dagger=\numParticles\delta_{x_0}\in\measp(D)
\qquad\text{and}\qquad
A\rho=p^\sct\|\rho\|\mathcal{L}\restr D+(1-p^\sct)G\ast\rho
\end{equation*}
for some $\numParticles\in\N$ and $x_0\in D$ and a smooth positron range kernel $G$ with unit integral and maximum at $0$.
Again, the measurement $E$ is a realization of $\Poi{A\rho^\dagger}$ and may for instance be given by $E=\numNonScattered\delta_{x_0}+\sum_{i=1}^{\numParticles-\numNonScattered}\delta_{x_i}$
for some $\numNonScattered\in\N$ and scattered detection positions $x_1,\ldots,x_{\numParticles-\numNonScattered}\in D$ (where typically $\numNonScattered/\numParticles\approx1-p^\sct$).
The corresponding negative logarithm of the likelihood (again taking $\nu=\mathcal{L}\restr D$) then reads
\begin{equation*}
\hat J^E(\rho)=\norm{\rho}-\int\log\left( \frac{\d A\rho}{\d\mathcal{L}} \right)\d E
=\norm\rho-\numNonScattered\log\left( \frac{\d A\rho}{\d\mathcal{L}}(x_0) \right)-\sum_{i=1}^{\numParticles-\numNonScattered}\log\left( \frac{\d A\rho}{\d\mathcal{L}}(x_i)\right).
\end{equation*}
If the $x_i$ are sufficiently far from each other with respect to the width of $G$,
it is not difficult to see that this is minimized by a measure of the form $\rho=\alpha\delta_{x_0}+\beta\sum_{i=1}^{\numParticles-\numNonScattered}\delta_{x_i}$ for some $\alpha,\beta>0$,
so all scatter events are explained by a nonscattered Dirac measure.

As discussed in \cref{rem:discreteVsPositronRange}, another reasonable setting to consider is the one in which instead of a positron range we have discrete measurements,
which is our last example.
To this end let $D$ be discretized into $M$ half-open intervals $\Gamma_1,\ldots,\Gamma_M$ of width $\frac1M$ and with centres $z_1,\ldots,z_M\in D$ and take
\begin{equation*}
\rho^\dagger=\numParticles\delta_{x_0}\in\measp(D)
\qquad\text{and}\qquad
A\rho=B_\discrete(p^\sct\|\rho\|\mathcal{L}\restr D+(1-p^\sct)\rho)
\qquad\text{with}\qquad
B_\discrete\mu=\sum_{i=1}^M\mu(\Gamma_i)\delta_{z_i}.
\end{equation*}
Assuming $x_0\in\Gamma_j$, a typical measurement would be $E=\numNonScattered\delta_{z_j}+\sum_{i\in I}\delta_{z_i}$ for some integer $\numNonScattered\approx(1-p^\sct)\numParticles$ and $I\subset\{1,\ldots,M\}$ with $\numParticles-\numNonScattered$ elements.
Taking $\nu=\sum_{i=1}^M\delta_{z_i}$ the corresponding negative logarithm of the likelihood reads
\begin{equation*}
\hat J^E(\rho)=\norm{\rho}-\int\log\left( \frac{\d A\rho}{\d\nu} \right)\d E
=\norm\rho-\numNonScattered\log\left( A\rho(\{z_j\}) \right)-\sum_{i\in I}\log\left( A\rho(\{z_i\})\right).
\end{equation*}
Again it is not difficult to see that this is minimized by a measure of the form $\rho=\alpha\delta_{z_j}+\beta\sum_{i\in I}\delta_{z_i}$ for some $\alpha,\beta>0$,
so again all scattered events are explained by a nonscattered Dirac measure.

\subsection{Mixed integer MAP estimate and convex relaxation}\label{sec:relaxation}
To mitigate the influence of the MAP estimate's bias we modify the considered functional
(note that with the conventions introduced in \cref{sec:discrete} we consider the continuous and discrete setting simultaneously).
To this end we introduce the new auxiliary random variable $\bm{E}^\sct$ which denotes the photon detections after a scatter event has occurred.
Analogously to \cref{sec:likelihood}, given a radioactive material density $\rho=\d t\otimes\rho_t=\d t\otimes\Mv_t\theta\in\measp([0,T]\times\dom)$, this random variable is given by
\begin{equation*}
\bm E^\sct=\Poi{p^\sct A^\sct\theta}=\Poi{\d t\otimes p^\sct A^\sct\rho_t}.
\end{equation*}
Accordingly, the events resulting from nonscattered detections are described by
\begin{equation*}
\bm{E}^\dt=\Poi{p^\dt A^\dt\theta}=\Poi{\d t\otimes p^\dt A^\dt\rho_t},
\end{equation*}
which conditioned on the realization $\theta$ of the random variable $\bm\theta$ is independent of $\bm{E}^\sct$.
The actual measurement is then a realization of the random variable $\bm{E}=\bm{E}^\dt+\bm{E}^\sct$.
With this splitting of the measurement we repeat the derivation of a MAP estimate, this time for $\theta_{t_1,\ldots,t_K}$ as well as for the realization $E^\sct$ of $\bm{E}^\sct$.
We employ the same base measure $\baseMeasMeas$ for $\bm E^\d$ and $\bm E^\sct$ (again the final MAP estimate will be independent of that choice).
Using the same notation as before, Bayes' rule yields
\begin{align*}
&\density{\bm\theta_{t_1,\ldots,t_K},\bm E^\sct|\bm E}(\theta_{t_1,\ldots,t_K}, E^\sct|E)
=\frac{\density{\bm\theta_{t_1,\ldots,t_K},\bm E^\sct, \bm E}(\theta_{t_1,\ldots,t_K}, E^\sct, E)}{\density{\bm E}(E)}\\
=&\frac{\density{\bm E^\sct, \bm E|\bm\theta_{t_0,\ldots,t_K}}(E^\sct, E|\theta_{t_1,\ldots,t_K})\density{\bm\theta_{t_1,\ldots,t_K}}(\theta_{t_1,\ldots,t_K})}{\density{\bm E}(E)}.
\end{align*}
The only new term is $\density{\bm E^\sct, \bm E|\theta_{t_0,\ldots,t_K}}(E^\sct, E|\theta_{t_1,\ldots,t_K})$,
and by the same argument as in \cref{sec:likelihood} it equals $\density{\bm E^\sct, \bm E|\bm\theta}(E^\sct, E|\theta)$.
Exploiting the conditional independence of $\bm{E}^\dt$ and $\bm{E}^\sct$ given $\theta$ we can calculate
\begin{equation*}
\density{\bm E^\sct, \bm E|\bm \theta_{t_1,\ldots,t_K}}(E^\sct, E|\theta_{t_1,\ldots,t_K})
=\density{\bm E^\sct, \bm E|\bm\theta}(E^\sct, E|\theta)
=\density{\bm E^\sct, \bm E^\dt|\bm\theta}(E^\sct, E-E^\sct|\theta)
=\density{\bm E^\sct|\bm\theta}(E^\sct|\theta)
\density{\bm E^\dt|\bm\theta}(E-E^\sct|\theta)
\end{equation*}
with
\begin{equation*}
\density{\bm E^i|\bm\theta}(\tilde E|\theta)
=\left(\prod_{(t,a,b)\in\tilde E}\RNderivative{(p^iA^i\theta)}{\nu}(t,a,b)\right)\exp\left( ||\nu||-||p^iA^i\theta|| \right)
\end{equation*}
for $i\in\{\sct,\dt\}$, just like in \cref{sec:likelihood}.

Repeating now the steps from \cref{sec:posterior} we finally arrive at the functional
\begin{equation*}
\hat J^E(\rho,\momentum,E^\sct)=
\|A\rho\|
-\sum_{(t,a,b)\in E^\sct}\log\left(\RNderivative{(p^\sct A^\sct\rho)}{\nu}(t,a,b)\right)
-\sum_{(t,a,b)\in E\setminus E^\sct}\log\left(\RNderivative{(p^\dt A^\dt\rho)}{\nu}(t,a,b)\right)
+\beta\BBEnergy(\rho,\momentum),
\end{equation*}
to be minimized for $\rho\in\measp(\spacetime)$, $\momentum\in\meas(\spacetime)^3$, and $E^\sct\subset E$.
So far, nothing is won, yet, the bias of the MAP estimate (the minimizer of this new functional) of course persists.
However, it is now easy to counteract this bias:
Since one knows that approximately the fraction $p^\sct$ of the total number $|E|$ of photon pair detections must actually have been scattered,
one may restrict the optimization in $E^\sct$ to subsets of $E$ with exactly this estimated number of elements.
Then the prescribed fraction of photon pair detections is no longer misinterpreted as unscattered events.
Actually, one would prescribe $|E^\sct|$ slightly larger than its expected value in order to get an acceptably small probability of underestimating the true value:
An underestimation would not sufficiently reduce the bias of hallucinating radioactive material where there is actually none,
while a slight overestimation is less problematic from the application viewpoint -- it will just slightly decrease the spatiotemporal accuracy of the reconstruction.

Unfortunately, the above suggested minimization in the set $E^\sct$ is a numerically difficult combinatorial problem.
Therefore we replace the mixed integer optimization problem
\begin{equation}\label{eqn:combinatorial}
\text{minimize }\hat J^E(\rho,\momentum,E^\sct)\text{ such that }E^\sct\subset E\text{ with }|E^\sct|=\numScattered
\end{equation}
with the continuous optimization problem
\begin{equation}\label{eqn:max}
\text{minimize }\bar J^{E,q}(\rho,\momentum)
=\norm{A\rho}
-\sum_{(t,a,b)\in E}\log\max\left\{q\RNderivative{(p^\sct A^\sct\rho)}{\nu}(t,a,b),\RNderivative{(p^\dt A^\dt\rho)}{\nu}(t,a,b)\right\}
+\beta\BBEnergy(\rho,\momentum)
\end{equation}
in which we now have to fix the parameter $q$ instead of $\numScattered$
(note that $\log q$ can be thought of like a Lagrange multiplier for the constraint $|E^\sct|=\numScattered$).
For a proper pairing of the tuning parameters $q$ and $\numScattered$ both optimization problems are indeed related as we will show below.
In more detail, for every $q$ we can find a $\numScattered$ such that \eqref{eqn:combinatorial} has the same solution as \eqref{eqn:max}.
The other direction does unfortunately not hold:
There may in principle exist values of $\numScattered$ such that \eqref{eqn:max} is not equivalent to \eqref{eqn:combinatorial} for any $q$.
However, the relation between $q$ and $\numScattered$ is at least monotone with several more desirable properties as we show below.

\begin{remark}[Existence of solutions]
	We do not prove existence of solutions to \eqref{eqn:combinatorial} or \eqref{eqn:max} here
	since they do not represent our final minimization problem.
	However, we note that both problems do admit solutions
	by almost exactly the same proof as for our existence result \cref{thm:ExistenceMinimizers} for the final minimization problem.
	At least this proof applies as long as $E^\sct$ in \eqref{eqn:combinatorial} is fixed; the subsequent optimization over $E^\sct$ then is just a finite optimization and thus well-posed.
\end{remark}

We first show that for every $q$ there is a corresponding $\numScattered$ with an equivalent optimization problem.
Essentially, any element of the measurement $E$ at which $q\RNderivative{(p^\sct A^\sct\rho)}{\nu}$ dominates $\RNderivative{(p^\dt A^\dt\rho)}{\nu}$ is interpreted as scattered,
therefore we introduce the notation
\begin{align*}
\overline E^\sct(q,\rho)&\textstyle=\{(t,a,b)\in E\,|\,q\RNderivative{(p^\sct A^\sct\rho)}{\nu}(t,a,b)\geq\RNderivative{(p^\dt A^\dt\rho)}{\nu}(t,a,b)\},\\
\underline E^\sct(q,\rho)&\textstyle=\{(t,a,b)\in E\,|\,q\RNderivative{(p^\sct A^\sct\rho)}{\nu}(t,a,b)>\RNderivative{(p^\dt A^\dt\rho)}{\nu}(t,a,b)\}\subset\overline E^\sct(q,\rho)
\end{align*}
for the maximal and minimal set of detections interpreted as scatter.

\begin{proposition}[Equivalence of minimization problems]\label{thm:debiasing}
	For every $q>0$ and every solution $(\rho,\momentum)$ of \eqref{eqn:max}
	there exists a $\numScattered\in\N$ and some $E^\sct\subset E$ such that $(\rho,\momentum,E^\sct)$ solves \eqref{eqn:combinatorial} and $\absNorm{E^\sct}=\numScattered$.
\end{proposition}
\begin{proof}
	Take any $E^\sct$ satisfying $\underline E^\sct(q,\rho)\subset E^\sct\subset\overline E^\sct(q,\rho)$ and $\numScattered=|E^\sct|$.
	Now let $(\tilde\rho,\tilde\momentum,\tilde E^\sct)$ be a competitor for \eqref{eqn:combinatorial} satisfying $|\tilde E^\sct|=\numScattered$, then 
	\begin{align*}
	\hat J^E(\tilde\rho,\tilde\momentum,\tilde E^\sct)&=\norm{A\tilde\rho}-\int\log\left( \RNderivative{(p^\sct A^\sct\tilde\rho)}{\nu} \right)\d \tilde E^\sct-\int\log\left( \RNderivative{(p^\dt A^\dt\tilde\rho)}{\nu} \right)\d(E-\tilde E^\sct)+\beta\BBEnergy(\tilde\rho,\tilde\momentum)\\
	&\textstyle=\norm{A\tilde\rho}-\int \log(q\RNderivative{(p^\sct A^\sct\tilde\rho)}{\nu})\,\d \tilde E^\sct-\int \log(\RNderivative{(p^\dt A^\dt\tilde\rho)}{\nu})\,\d (E-\tilde E^s)+\numScattered\log(q)+\beta\BBEnergy(\tilde\rho,\tilde\momentum)\\
	&\ge \bar{J}^{E,q}(\tilde\rho,\tilde\momentum)+\numScattered\log q\\
	&\ge \bar{J}^{E,q}(\rho,\momentum)+\numScattered\log q\\
	&\textstyle=\norm{A\rho}-\int \log(q\RNderivative{(p^\sct A^\sct\rho)}{\nu})\,\d  E^\sct-\int \log(\RNderivative{(p^\dt A^\dt\rho)}{\nu})\,\d (E- E^\sct)+\numScattered\log(q)+\beta\BBEnergy(\rho,\momentum)\\
	&=\hat{J}^E(\rho,\momentum,E^\sct).
	\qedhere
	\end{align*}
\end{proof}

Even though the map assigning a number $\numScattered$ of scattered events to a tuning parameter $q$ may not be surjective
(that is, some $\numScattered$ may not be reached by any $q$),
it is still monotonically increasing (and therefore, due to the discreteness of $\numScattered$ also piecewise constant) as we show below.
This means that one can readily tune the number $\numScattered$ of measurements interpreted as scatter by in- or decreasing $q$ and that this number $\numScattered$ is robust to changes in $q$.
To state the result let us abbreviate by
\begin{align*}
\overline \numScattered(q)&=\max\big\{|\overline E^\sct(q,\rho)|\,\big|\,(\rho,\momentum)\text{ minimizes }\bar J^{E,q}\big\},\\
\underline \numScattered(q)&=\min\,\big\{|\underline E^\sct(q,\rho)|\,\big|\,(\rho,\momentum)\text{ minimizes }\bar J^{E,q}\big\}\leq\overline \numScattered(q)
\end{align*}
the maximum and minimum number of events that can be interpreted as scatter in a solution of \eqref{eqn:max}
(and thus by \cref{thm:debiasing} equivalently in a solution of \eqref{eqn:combinatorial}).

\begin{proposition}[Scatter interpretations as function of tuning parameter]\label{thm:pProperties}\hfill
	\begin{enumerate}
		\item
		There exists a monotonically increasing, piecewise constant, integer-valued function $\numScattered:[0,\infty)\to\N$ with $\numScattered(0)=0$ and $\numScattered(q)=|E|$ for all large enough $q$ such that
		$\overline \numScattered$ is the upper and $\underline \numScattered$ the lower semi-continuous envelope of $\numScattered$.
		
		\item
		The function $q\mapsto\min\bar J^{E,q}$ is non-increasing and continuous. Any minimizer $(\rho,\momentum)$ of $\bar J^{E,q}$ also minimizes $\bar J^{E,\tilde q}$ for any $\tilde q$ with $|\underline E^\sct(q,\rho)|\leq \numScattered(\tilde q)\leq|\overline E^\sct(q,\rho)|$.
	\end{enumerate}
\end{proposition}
\begin{proof}
	The first statement is an immediate consequence of the following three properties, which we will subsequently prove.
	\begin{enumerate}
		\item[(a)]\label{enm:monotonicity}
		$q_1<q_2$ implies $\overline \numScattered(q_1)\leq\underline \numScattered(q_2)$.
		\item[(b)]\label{enm:semi-continuity}
		$\overline \numScattered(q)$ is upper and $\underline \numScattered(q)$ lower semi-continuous in $q$.
		\item[(c)]\label{enm:limits}
		$\underline \numScattered(0)=0$ and $\overline \numScattered(q)=|E|$ for all $q$ large enough.
	\end{enumerate}
	(a) Let $q_1<q_2$ and assume $\overline \numScattered(q_1)>\underline \numScattered(q_2)$, where the maximum and minimum are realized by $(\rho_1,\momentum_1)$ and $(\rho_2,\momentum_2)$, respectively. Then
	\begin{align*}
	\bar{J}^{E,q_2}(\rho_1,\momentum_1)
	&\le  \norm{A\rho_1} -\int\log\left(q_2\RNderivative{(p^\sct A^\sct\rho_1)}{\nu}\right)\d\overline E^\sct(q_1, \rho_1)\\
	 &\quad\quad\quad-\int\log\left(\RNderivative{(p^\dt A^\dt\rho_1)}{\nu}\right)\d(E-\overline E^\sct(q_1,\rho_1)) +\beta\BBEnergy(\rho_1,\momentum_1)\\
	&= \bar J^{E,q_1}(\rho_1,\momentum_1)-\overline \numScattered(q_1)\log\frac{q_2}{q_1}\\
	&\le \bar J^{E,q_1}(\rho_2,\momentum_2)-\overline \numScattered(q_1)\log\frac{q_2}{q_1}\\
	&< \bar J^{E,q_1}(\rho_2,\momentum_2)-\underline \numScattered(q_2)\log\frac{q_2}{q_1}\\
	&\le-\underline \numScattered(q_2)\log\frac{q_2}{q_1}\!+\!\norm{A\rho_2}
	\!-\!\int\!\!\log\left(q_1\RNderivative{(p^\sct A^\sct\rho_2)}{\nu}\right)\d\underline E^\sct(q_2,\rho_2)
	\!\\
	&\quad\quad\quad-\! \int\!\!\log\left(\RNderivative{(p^\dt A^\dt\rho_2)}{\nu}\right)\d(E\!-\!\underline E^\sct(q_2,\rho_2))
	\!+\! \beta\BBEnergy(\rho_2,\momentum_2)\\
	&=\bar{J}^{E,q_2}(\rho_2,\momentum_2)
	\end{align*}
	contradicts the fact that $(\rho_2,\momentum_2)$ minimizes $\bar J^{E,q_2}$.\\
	(b) We show upper semi-continuity of $\overline \numScattered(q)$; lower semi-continuity of $\underline \numScattered(q)$ follows analogously. Let $q_n\to q$ and denote the corresponding minimizers from the definition of $\overline \numScattered(q_n)$ by $(\rho_n,\momentum_n)$. Without loss of generality we may assume $\limsup_{n\to\infty}\overline \numScattered(q_n)=\lim_{n\to\infty}\overline \numScattered(q_n)$ (else just pass to a sub-sequence). Due to $\bar{J}^{E,q_n}(\rho_n,\momentum_n)\le\bar{J}^{E,q_n}(\tilde\rho,0)\le\bar J^{E,0}(\tilde\rho, 0)<\infty$ for $\tilde\rho=\mathcal L^4\restr(\spacetime)$ the corresponding energies are uniformly bounded from which we can derive (\cref{thm:coercivity}) that the total variations $\norm{\rho_n}$ and $\norm{\momentum_n}$ are uniformly bounded.
	Consequently, there exists a weakly-* converging subsequence (still indexed by $n$) such that $(\rho_n,\momentum_n)\weakstarto(\rho,\momentum)$.
	Now in \cref{LemmaContinuityForwardOp} we will show continuity properties of the forward operator;
	in particular, we will show that the boundedness of $\BBEnergy(\rho_n,\momentum_n)$ and the weak-* convergence $(\rho_n,\momentum_n)\weakstarto(\rho,\momentum)$ imply
	uniform convergence $\RNderivative{(p^i A^i\rho_n)}{\nu}\to\RNderivative{(p^i A^i\rho)}{\nu}$ as $n\to\infty$ for $i\in\{\sct,\dt\}$.
	Together with the weak lower semi-continuity of $\BBEnergy$ this implies
	\begin{equation*}
	\bar J^{E, q}(\rho,\momentum)
	\leq\liminf_{n\to\infty}\bar J^{E,q_n}(\rho_n,\momentum_n)
	\le\lim_{n\to\infty}\bar J^{E,q_n}(\tilde\rho,\tilde\momentum)
	=\bar J^{E,q}(\tilde\rho,\tilde\momentum)
	\end{equation*}
	for any competitor $(\tilde\rho, \tilde\momentum)$, thus $(\rho,\momentum)$ minimizes $\bar J^{E, q}$.
	Furthermore, $\rho_n\weakstarto\rho$ implies \\  $\limsup_{n\to\infty}\overline E^\sct(q_n,\rho_n)=\overline E^\sct(q,\rho)$ because by continuity of the forward operator \cref{LemmaContinuityForwardOp}
	\begin{align*}
\limsup_{n\to\infty}\mathds{1}_{\overline E^\sct(q_n,\rho_n)}=\limsup_{n\to\infty}\mathds{1}_{\left\{q_n\RNderivative{A^\sct\rho_n}{\nu}(t,a,b)\ge\RNderivative{A^{\dt}\rho_n}{\nu}(t,a,b)\right\}}=\mathds{1}_{\overline E^\sct(q,\rho)}.
	\end{align*}
	Therefore
	\begin{equation*}
	\overline \numScattered(q)\geq|\overline E^\sct(q,\rho)|\geq\liminf_{n\to\infty}|\overline E^\sct(q_n,\rho_n)|=\lim_{n\to\infty}\overline \numScattered(q_n).
	\end{equation*}
	(c) $\underline \numScattered(0)=0$ follows from $\underline E^\sct(0,\rho)=\emptyset$ for all $\rho$.
	To show $\overline \numScattered(q)=|E|$ for $q$ large enough, assume to the contrary that there is an increasing sequence $q_n\to\infty$ with corresponding minimizers $(\rho_n,\momentum_n)$ of $\bar J^{E,q_n}$ and points $(t_n, a_n, b_n)\in E$ where $q_n\RNderivative{(p^\sct A^\sct\rho)}{\nu}(t,a,b)<\RNderivative{(p^\dt A^\dt\rho)}{\nu}(t,a,b)$.
	Again, $\bar J^{E,q_n}(\rho_n,\momentum_n)\le \bar J^{E,0}(\mathcal{L}^4,0)$ is uniformly bounded giving rise to a weakly-* converging subsequence $\rho_n\weakstarto\rho$.
	By continuity of the forward operator we thus have
	$$\RNderivative{(p^\sct A^\sct\rho_n)}{\nu}-\RNderivative{(p^\dt A^\dt\rho_n)}{\nu}/q_n\to\RNderivative{(p^\sct A^\sct\rho)}{\nu}$$
	uniformly as $n\to\infty$.
	Consequently, $\RNderivative{(p^\sct A^\sct\rho_n)}{\nu}(t_n,a_n,b_n)\leq\RNderivative{(p^\dt A^\dt\rho_n)}{\nu}(t_n,a_n,b_n)/q_n$ for arbitrary $n$ can only hold for $\rho=0$.
	The nonnegativity of $\rho_n$ therefore implies $\rho_n\to 0$ strongly and hence
	\begin{align*}
	-\log\max\left\{q_n\RNderivative{(p^\sct A^\sct\rho_n)}{\nu}(t_n, a_n, b_n),\RNderivative{(p^\dt A^\dt\rho_n)}{\nu}(t_n, a_n, b_n)\right\}
	=&-\log\RNderivative{(p^\dt A^\dt\rho_n)}{\nu}(t_n, a_n, b_n)\\
	\geq&-\log(C\norm{\rho_n})\to\infty
		\end{align*}
	as $n\to\infty$.
	This in turn implies $\bar J^{E,q_n}(\rho_n,\momentum_n)\to\infty$ contradicting the uniform boundedness of $\bar J^{E, q_n}(\rho_n,\momentum_n)$.\newline
	
	As for the second statement, let $|\underline E^\sct(q,\rho)|\leq \numScattered(\tilde q)\leq|\overline E^\sct(q,\rho)|$
	and let $(\tilde\rho,\tilde\momentum)$ minimize $\bar J^{E,\tilde q}$ with $|\underline E^\sct(\tilde q,\tilde\rho)|\leq \numScattered(\tilde q)\leq|\overline E^\sct(\tilde q,\tilde\rho)|$ (such a minimizers exists by definition of $\numScattered$). Then
	\begin{equation*}
	\bar J^{E,\tilde q}(\rho,\momentum)
	\leq\bar J^{E,q}(\rho,\momentum)-\numScattered(\tilde q)\log\frac{\tilde q}q
	\leq\bar J^{E,q}(\tilde\rho,\tilde\momentum)-\numScattered(\tilde q)\log\frac{\tilde q}q
	\leq\bar J^{E,\tilde q}(\tilde\rho,\tilde\momentum)
	\end{equation*}
	so that $(\rho,\momentum)$ minimizes $\bar J^{E,\tilde q}$.
	Finally, we show that $q\mapsto\min\bar J^{E,q}$ is nonincreasing and continuous. Since for $(\rho,\momentum)\in\measp(\spacetime)\times\meas(\spacetime)^3$ the functions $q\mapsto\bar J^{E,q}(\rho,\momentum)$ are not Lipschitz, we cannot conclude continuity of the pointwise minimum. Instead, we consider $r\mapsto\min\bar J^{E,\exp(r)}$ which is the pointwise infimum over all maps $r\mapsto\bar J^{E,\exp(r)}(\rho,\momentum)$ for $(\rho,\momentum)\in\measp(\spacetime)\times\meas(\spacetime)^3$.
	Since each of these maps is nonincreasing and Lipschitz continuous with Lipschitz constant bounded by $|E|$, so is their pointwise infimum.
	As a consequence, $q\mapsto\min\bar J^{E,q}$ is nonincreasing and continuous.
\end{proof}

As a last simplifying step we convexify the energy $\bar J^{E,q}$ by replacing $-\log\max\{a,b\}$ with its convex envelope $-\log(a+b)$. We thus arrive at our final reconstruction functional
\begin{equation}\label{eqn:finalFunctional}
J^{E,q}(\rho,\momentum)=\norm{A\rho}-\int\log\left(\RNderivative{(\Aq{q}\rho)}{\nu}\right)\d E+\beta\BBEnergy(\rho,\momentum)
\qquad\text{with}\qquad
\Aq{q}=qp^\sct A^\sct+p^\dt A^\dt,
\end{equation}
where by convention we set $J^{E,q}(\rho,\momentum)=\infty$ if $\rho\geq0$ or \eqref{eqn:continuity} are violated.\\
As we will verify in exemplary calculations below, choosing $q$ appropriately indeed removes the bias.

\subsection{Bias removal by sufficiently high tuning parameter}\label{sec:biasRemoval}
We resume the two PET-like examples from \cref{sec:bias} showing that our new reconstruction functional effectively removes the bias.

First recall the setting on the periodic domain $D=[0,1]$ with ground truth, forward operator, and measurement
\begin{equation*}
\rho^\dagger=\numParticles\delta_{x_0},
\qquad
A\rho=p^\sct\|\rho\|\mathcal{L}\restr D+(1-p^\sct)G\ast\rho,
\qquad\text{and}\qquad
E=\numNonScattered\delta_{x_0}+\sum_{i=1}^{\numParticles-\numNonScattered}\delta_{x_i}.
\end{equation*}
The modified reconstruction functional in this setting reads
\begin{align*}
J^{E,q}(\rho)
&=\norm{\rho}-\int\log\left(qp^\sct\|\rho\|+(1-p^\sct)G\ast\rho\right)\d E\\
&=\norm\rho-\numNonScattered\log\left(qp^\sct\|\rho\|+(1-p^\sct)(G\ast\rho)(x_0)\right)-\sum_{i=1}^{\numParticles-\numNonScattered}\log\left(qp^\sct\|\rho\|+(1-p^\sct)(G\ast\rho)(x_i)\right).
\end{align*}
Again, if the $x_i$ are sufficiently far from each other with respect to the width of $G$,
this is minimized by a measure of the form $\rho=\alpha\delta_{x_0}+\beta\sum_{i=1}^{\numParticles-\numNonScattered}\delta_{x_i}$, for which $(G\ast\rho)(x_0)=G(0)\alpha$ and $(G\ast\rho)(x_i)=G(0)\beta$.
It is now straightforward to check via the optimality conditions of minimizing $J^{E,q}$ for $\alpha,\beta\geq0$
that the unique solution is given by $\alpha=\numParticles$ and $\beta=0$ or equivalently $\rho=\rho^\dagger$ if and only if the tuning parameter is chosen sufficiently large,
\begin{equation*}
q\geq\frac{(1-p^\sct)G(0)}{p^\sct(\numNonScattered-1)}.
\end{equation*}
Essentially, this confirms the expectation that the tuning parameter $q$ has to be so large
that at every scattered measurement, $q$ times the scatter part of the forward operator must dominate the nonscatter part of a hallucinated Dirac measure.
Then scattered events are no longer interpreted as nonscattered.
Of course the above calculation breaks down if the total mass $\numParticles$ is so large
that the $p^\sct \numParticles$ many scattered events come within distance of the positron range kernel diameter of each other (this diameter roughly behaves like $1/G(0)$);
in that case $\frac{(1-p^\sct)G(0)}{p^\sct(\numNonScattered-1)}\lesssim1$ and the parameter $q$ is not needed (meaning that it can be set to one).

In the second, discrete measurement setting we used
\begin{equation*}
\rho^\dagger=\numParticles\delta_{x_0},
\qquad
A\rho=B_\discrete(p^\sct\|\rho\|\mathcal{L}\restr D+(1-p^\sct)\rho),
\qquad\text{and}\qquad
E=\numNonScattered\delta_{z_j}+\sum_{i\in I}\delta_{z_i}
\end{equation*}
for $z_i$ the centres of the $M$ discrete detector intervals, $I\subset\{1,\ldots,M\}$ with $\numParticles-\numNonScattered$ elements, and $x_0\in\Gamma_j$.
The modified reconstruction functional in this setting reads
\begin{align*}
J^{E,q}(\rho)
&=\norm{\rho}-\int\log\left(\RNderivative{\Aq{q}\rho}{\nu}\right)\d E\\
&=\norm\rho-\numNonScattered\log\left(\tfrac{qp^\sct\|\rho\|}M+(1-p^\sct)\rho(\Gamma_j)\right)-\sum_{i\in I}\log\left(\tfrac{qp^\sct\|\rho\|}M+(1-p^\sct)\rho(\Gamma_i)\right).
\end{align*}
It is readily seen that this is minimized by a measure of the form $\rho=\alpha\delta_{x_0}+\beta\sum_{i\in I}\delta_{z_i}$
(of course, the Dirac masses may also be arbitrarily shifted within each $\Gamma_i$, since the reconstruction functional is oblivious to the exact position within $\Gamma_i$).
Again by checking the optimality conditions we obtain that $(\alpha,\beta)=(\numParticles,0)$ or equivalently $\rho=\rho^\dagger$ minimizes $J^{E,q}$ if and only if
\begin{equation*}
q\geq\frac{(1-p^\sct)M}{p^\sct(\numNonScattered-1)}.
\end{equation*}
Compared to the previous case with positron range, the positron length scale $1/G(0)$ was simply replaced with the detector length scale $1/M$.
Again, the bias to interpret scattered events as nonscattered is removed, if $q$ times the scatter part of the forward operator dominates the nonscatter part of a hallucinated Dirac measure.
If the total mass is large enough to fill every detector with scatter events, $q$ is again no longer needed and can be set to one.

These observations can be summarized in the following heuristic for model \eqref{eqn:finalFunctional}.
\begin{remark}[Heuristic for choice of $q$]\label{rem:heuristicDebiasing}
	In the discrete setting, the probability of a scattered photon pair being detected in the detector pair $(\Gamma_i,\Gamma_j)$ during time interval $\tau_k$
	is $p^\sct A^\sct\rho^\dagger=\Delta Tp^\sct\norm{\rho^\dagger}\hd^2(\Gamma_i)\hd^2(\Gamma_j)/\hd^2(\domDelta)^2$.
	On the other hand the probability that a nonscattered photon pair is detected in $(\Gamma_i,\Gamma_j)$,
	which emanated from some Dirac mass $\delta_x$ at a point $x$ in between $\Gamma_i$ and $\Gamma_j$,
	is roughly $\Delta Tp^\dt\min\{\hd^2(\Gamma_i)/\mathrm{dist}(x,\Gamma_i),\hd^2(\Gamma_j)^2/\mathrm{dist}(x,\Gamma_j)^2\}$.
	Now $q$ should be chosen larger than the ratio between the latter and the former probability,
	\begin{equation*}
	q\geq\max_{i\neq j,x\in\dom}\frac{p^\dt\hd^2(\domDelta)^2}{p^\sct\norm{\rho^\dagger}\hd^2(\Gamma_i)\hd^2(\Gamma_j)}
	\min\left\{\frac{\hd^2(\Gamma_i)}{\mathrm{dist}(x,\Gamma_i)^2},\frac{\hd^2(\Gamma_j)}{\mathrm{dist}(x,\Gamma_j)^2}\right\}.
	\end{equation*}
	In addition we should pick $q\geq1$ (recall that $q=1$ yields the original MAP estimate).
	For $M$ detectors of equal area $\hd^2(\partial\domDelta)/M$ this becomes
	\begin{equation*}
	q\geq\max\left\{1,\frac{p^\dt M}{p^\sct\norm{\rho^\dagger}}\frac{\hd^2(\domDelta)}{\delta^2}\right\}.
	\end{equation*}
	In the continuous setting, using an analogous reasoning, the number $M$ of detectors simply has to be replaced with $G(0)$, the maximum of the positron range kernel, yielding
	\begin{equation*}
	q\geq\max\left\{1,\frac{p^\dt G(0)}{p^\sct\norm{\rho^\dagger}}\frac{\hd^2(\domDelta)}{\delta^2}\right\}.
	\end{equation*}
\end{remark}

\section{Model properties}
In this section we show existence of minimizers to \eqref{eqn:finalFunctional}, derive the dual optimization problem, and prove some scale invariances of the reconstruction problem.

\subsection{Existence of minimizers and properties of forward operator}\label{sec:existence}
As a preparation to prove existence of minimizers to $J^{E,q}$ we require some continuity properties of the forward operator.
To this end it is convenient to rewrite the detection part $A^\dt$ with the help of the so-called X-ray transform:
For $\theta\in S^2$ a vector in the unit sphere define
\begin{equation*}
\theta^\perp=\{s\in\R^3\,|\,s\cdot\theta=0\}
\end{equation*}
to be the orthogonal complement of $\theta$,
and let $\pi_{\theta^\perp}:\R^3\to\theta^\perp$ denote the orthogonal projection onto $\theta^\perp$.
The X-ray transform is then defined as
\begin{align}\label{definitionXRayTransform}
P:L^1(\domDeltaHalf)\to L^1(\mathcal C),
\quad
Pf(\theta,s)=\int_{\{r\in\R\,|\,s+r\theta\in\domDeltaHalf\}} f(s+r\theta)\,\d\mathcal L(r),\\
\qquad\text{where }
\mathcal C=\{(\theta,s)\in S^2\times\R^3\,|\,s\in\pi_{\theta^\perp}(\domDeltaHalf)\}\nonumber.
\end{align}
Note that the X-ray transform satisfies the symmetry $Pf(\theta,s)=Pf(-\theta,s)$. 
On $\mathcal C$ we will use the Borel measure $\hd^2\otimes\hd^2$, defined by dual pairing with any continuous function $f:\mathcal C\to\R$ as
\begin{equation*}
\int_{\mathcal C}f\,\d(\hd^2\otimes\hd^2)
=\int_{S^2}\int_{\theta^\perp}f(\theta,s)\,\d\hd^2(s)\,\d\hd^2(\theta).
\end{equation*}

Furthermore, the convolution $G\ast\lambda$ of some $\lambda\in\measp(\dom)$ with the continuous positron range kernel $G$ is absolutely continous with respect to $\mathcal L^3$
(it is even a continuous function) and will therefore be identified with its $\mathcal L^3$-density so that we may for instance write $P[G\ast\lambda]$.

\begin{lemma}[Scatterless detection operator]\label{thm:scatterlessDetection}
In the nondiscrete setting, for any $\lambda\in\measp(\dom)$ we have $A^\dt\lambda=\pushforward{R_\mathcal{C}}{[\frac1{4\pi}P[G\ast\lambda]\cdot(\hd^2\otimes\hd^2)]}$ with 
\begin{align*}
	R_\mathcal{C}\colon\mathcal{C}\to\boundDomDelta, \quad(\theta,x)\mapsto R(x,\theta).
\end{align*}
\end{lemma}
\begin{proof}
Let 
\begin{equation*}
F\colon \domDeltaHalf\times S^2\to\mathcal C,
\qquad(x, v)\mapsto(v, \pi_{v^\perp}(x)).
\end{equation*}
This way, it holds $R=R_\mathcal{C}\circ F$. Thus, for any measure $\mu\in\measp(\domDeltaHalf\times S^2)$ we have $\pushforward{R}\mu=\pushforward {R_\mathcal{C}}{\pushforward{F}\mu}$. The scatterless detection operator is then given by
\begin{equation*}
A^\dt\lambda
=\pushforward{R}{(B_{\mathrm{lines}} B_{\mathrm{pr}}\lambda)}
=\pushforward{R_\mathcal{C}}{\pushforward{F}{(B_{\mathrm{lines}} B_{\mathrm{pr}}\lambda)}}
=\pushforward{R_\mathcal{C}}{\pushforward{F}{(G\ast\lambda)\otimes\vol_{S^2}}}.
\end{equation*}
The desired result now follows from the straightforward relation
\begin{equation*}
\frac1{4\pi}P(f)\cdot(\hd^2\otimes\hd^2)
=\pushforward{F}{((f\mathcal L^3\restr\domDeltaHalf)\otimes\vol_{S^2})}
\end{equation*}
for any $f\in L^1(\domDeltaHalf)$.
\end{proof}

Next we estimate the density of the scatterless detection intensity with respect to the Hausdorff measure on $\boundDomDelta$.
To this end, for $(a,b)\in\boundDomDelta$ we abbreviate
\begin{equation*}
\theta(a,b)=\frac{b-a}{|b-a|}\in S^2,\qquad
s(a,b)=\pi_{\theta(a,b)^\perp}(a).
\end{equation*}

\begin{lemma}[Density of scatterless detection]\label{LemmaForwardDensitySplitting}
In the nondiscrete setting, for any $\lambda\in\measp(\dom)$ we have
\begin{equation*}
\RNderivative{A^\dt\lambda}{(\hd^2\restr\partial\domDelta)\otimes(\hd^2\restr\partial\domDelta)}(a,b)
=g(a,b)P[G\ast\lambda](\theta(a,b),s(a,b))
\end{equation*}
for some bounded smooth function $g:\boundDomDelta\to(0,\infty)$.
\end{lemma}
\begin{proof}
We aim to apply the transformation rule for integrals on Lipschitz manifolds (see for instance \cite[\S\,3.2.5, \S\,3.2.22]{FedererGMT}, \cite[\S\,3.3.2]{EvGa15}, \cite[Thm.\,2.71]{AmFuPa00}).
This transformation rule involves the Jacobian $JT$ of a Lipschitz map $T$ between Lipschitz manifolds $X,Y$, which is Hausdorff-almost everywhere defined as
\begin{equation*}
JT=\sqrt{\det(DT^*DT)}
\end{equation*}
with $DT:TX\to TY$ the differential of $T$ (a linear operator between the tangent spaces to $X$ and $Y$) and $DT^*$ its adjoint.

We first note that $\mathcal C$ is a smooth four-dimensional manifold embedded in $\R^6$ and thus has $\hd^4\restr\mathcal C$ as its volume measure.
However, above we employed the measure $\hd^2\otimes\mathcal L^2$ on $\mathcal C$.
We now show
\begin{equation*}
\RNderivative{(\hd^2\otimes\mathcal L^2)\restr\mathcal C}{\hd^4\restr\mathcal C}(\theta,s)
=\frac1{\sqrt{1+|s|^2}}.
\end{equation*}
To this end let $e_1,e_2,e_3$ denote the standard Euclidean basis vectors of $\R^3$ and define $T:S^2\times\R^2\to S^2\times\R^3$, $T(\theta,x)=(\theta,R_\theta{x\choose 0})$, where
\begin{equation*}
R_\theta=I+(\theta e_3^T-e_3\theta^T)+\frac1{1+e_3\cdot\theta}(\theta e_3^T-e_3\theta^T)^2
\end{equation*}
is the smallest three-dimensional rotation of $e_3$ onto $\theta$
(it obviously leaves vectors orthogonal to $e_3$ and $\theta$ invariant, and it can readily be checked that $R_\theta e_3=\theta$).
Note that $R_\theta e_1,R_\theta e_2$ represents an orthonormal basis of the tangent space $T_\theta S^2$ to $S^2$ in $\theta$.
It is straightforward to calculate that
in the basis $R_\theta e_1,R_\theta e_2,e_1,e_2$ of $T_{(\theta,x)}(S^2\times\R^2)$
and the basis $R_\theta e_1,R_\theta e_2,e_1,e_2,e_3$ of $T_{(\theta,x)}(S^2\times\R^3)$
the differential $DT(\theta,x)$ has the representation
\begin{equation*}
DT(\theta,x)=\left(\begin{smallmatrix}I&0\\C&R_\theta(e_1|e_2)\end{smallmatrix}\right),
\end{equation*}
where $I\in\R^{2\times2}$ denotes the identity matrix and $C=(\partial_\theta(R_\theta{x\choose0})(R_\theta e_1)|\partial_\theta(R_\theta{x\choose0})(R_\theta e_2))$.
Abbreviating $\tilde C=(e_1|e_2)^TR_\theta^TC$ we calculate
\begin{align*}
\det(DT^*DT)
&=\det\left(\begin{smallmatrix}I+C^TC&\tilde C^T\\\tilde C&I\end{smallmatrix}\right)\\
&=\det(I)\det(I+C^TC-\tilde C^TI^{-1}\tilde C)\\
&=\det(I+C^TC-(C^T-R_\theta(0|0|e_3)C^T)(C^T-R_\theta(0|0|e_3)C^T)^T)\\
&=\det(I+C^TC-C^T(I-\theta\otimes\theta)C)\\
&=\det(I+(C^T\theta)\otimes(C^T\theta))\\
&=1+|C^T\theta|^2.
\end{align*}
After a few tedious but straightforward steps of calculation one obtains $C^T\theta=(e_1|e_2)^TR_\theta^T(\frac{e_3\otimes\theta}{1+e_3\cdot\theta}-I)x$
as well as $|C^T\theta|^2=|x|^2$.
Now for any continuous function $f:\R^3\times\R^3\to\R$, by \cite[\S\,3.2.22]{FedererGMT} we have
\begin{multline*}
\int_{\{(\theta,s)\in S^2\times\R^3\,|\,s\in\theta^\perp\}}f\frac1{JT\circ T^{-1}}\,\d\hd^4
=\int_{S^2\times\R^2}f\circ T\,\d\hd^4
=\int_{S^2}\int_{\R^2}f\circ T(\theta,x)\,\d\mathcal L^2(x)\,\d\hd^2(\theta)\\
=\int_{S^2}\int_{\theta^\perp}f(\theta,s)\,\d\mathcal L^2(s)\,\d\hd^2(\theta)
=\int_{\{(\theta,s)\in S^2\times\R^3\,|\,s\in\theta^\perp\}}f\,\d(\hd^2\otimes\mathcal L^2),
\end{multline*}
therefore, as desired,
\begin{equation*}
\RNderivative{(\hd^2\otimes\mathcal L^2)\restr\mathcal C}{\hd^4\restr\mathcal C}(\theta,s)
=\frac1{JT\circ T^{-1}(\theta,s)}
=\frac1{\sqrt{\det(DT^*DT)}(\theta,R_\theta^Ts)}
=\frac1{\sqrt{1+|R_\theta^Ts|^2}}
=\frac1{\sqrt{1+|s|^2}}.
\end{equation*}

We now calculate the density of $A^\dt\lambda$ with respect to $(\hd^2\restr\partial\domDelta)\otimes(\hd^2\restr\partial\domDelta)$.
To this end let us introduce the diagonal $\Delta=\{(a,a)\,|\,a\in\R^3\}$ of $\R^3\times\R^3$ and the map
\begin{equation*}
\bar R:\boundDomDelta\setminus\Delta\to\{(\theta,s)\in S^2\times\R^3\,|\,s\in\theta^\perp\},\qquad
(a,b)\mapsto(\theta(a,b),s(a,b))
\end{equation*}
(which can be thought of as the inverse of $R_\mathcal{C}$ from \cref{thm:scatterlessDetection} and in fact is the inverse when restricting it to the range of $R_\mathcal{C}$).
Now consider an arbitrary continuous function $f:\boundDomDelta\to\R$.
By \cref{thm:scatterlessDetection} and \cite[\S\,3.2.22]{FedererGMT} we have
\begin{align*}
\int_{R_\mathcal{C}(\mathcal C)}f\,\d A^\dt\lambda
&=\int_{\mathcal C}f\circ 
R_\mathcal{C}\,\frac1{4\pi}P[G\ast\lambda]\,\d(\hd^2\otimes\mathcal L^2)\\
&=\int_{\mathcal C}f\circ R_\mathcal{C}(\theta,s)\,\frac1{4\pi}P[G\ast\lambda](\theta,s)\frac1{\sqrt{1+|s|^2}}\,\d\hd^4(\theta,s)\\
&=\int_{R_\mathcal{C}(\mathcal C)}f(a,b)\,\frac1{4\pi}P[G\ast\lambda](\theta(a,b),s(a,b))\frac1{\sqrt{1+|s(a,b)|^2}}J\bar R(a,b)\,\d\hd^4(a,b)\\
&=\int_{R_\mathcal{C}(\mathcal C)}f(a,b)\,\frac1{4\pi}P[G\ast\lambda](\theta(a,b),s(a,b))\frac1{\sqrt{1+|s(a,b)|^2}}J\bar R(a,b)\,\d\hd^2(a)\otimes\hd^2(b).
\end{align*}
The claim therefore holds with $g(a,b)=J\bar R(a,b)/(4\pi\sqrt{1+|s(a,b)|^2})$,
and it remains to show that $J\bar R$ is smooth.
However, this is a direct consequence of $\bar R$ being smooth on $R_\mathcal{C}(\mathcal C)$:
It can even be extended to a smooth map on $(\R^3\times\R^3)\setminus\Delta$,
and since $\partial\domDelta$ has at least distance $\delta/2$ from $\domDeltaHalf$, $R_\mathcal{C}(\mathcal C)$ stays bounded away from $\Delta$.

\end{proof}

This readily allows to prove the following boundedness result.

\begin{lemma}[Boundedness of forward operator]\label{LemmmaBoundednessForwardOp}
There exists a constant $C>0$ such that in the nondiscrete setting, for any $\lambda\in\measp(\dom)$ we have
\begin{equation*}
\RNderivative{A^\dt\lambda}{(\hd^2\restr\partial\domDelta)\otimes(\hd^2\restr\partial\domDelta)}
\leq C\norm\lambda.
\end{equation*}
As a consequence, for $q>0$ there exists $C>0$ such that for any $\rho\in\measp(\spacetime)$ satisfying \eqref{eqn:continuity} we have
\begin{equation*}
\tfrac1C\norm\rho
\leq\RNderivative{\Aq{q}\rho}{\nu}
\leq C\norm\rho.
\end{equation*}
in the discrete and nondiscrete setting (in the former even for vanishing positron range, that is, $G$ a Dirac).
\end{lemma}
\begin{proof}
By \cref{LemmaForwardDensitySplitting} we have
\begin{multline*}
\RNderivative{A^\dt\lambda}{(\hd^2\restr\partial\domDelta)\otimes(\hd^2\restr\partial\domDelta)}(a,b)
\leq\hat CP[G\ast\lambda](\theta(a,b),s(a,b))
=\hat C\int_\R(G\ast\lambda)(s(a,b)+r\theta(a,b))\,\d\mathcal L(r)\\
=\hat C\int_\R\int_\dom G(s(a,b)+r\theta(a,b)-x)\,\d\lambda(x)\,\d\mathcal L(r)
\leq\hat C\tilde C\lambda(\dom)
=\hat C\tilde C\norm\lambda
\end{multline*}
for $\hat C$ and $\tilde C$ the supremum norm of $g$ and $G$, respectively.
As for the second statement, since $\rho$ satisfies \eqref{eqn:continuity}, by \cite[Lemma 1.1.2]{chizat_unbalancedOT} we have
\begin{equation*}
\rho=\d t\otimes\rho_t
\qquad\text{and}\qquad
\norm{\rho_t}
=\rho_t(\dom)
=\frac1T\rho(\spacetime)
=\frac1T\norm\rho
\quad\text{for almost all }t\in[0,T].
\end{equation*}
Therefore, in the nondiscrete setting we obtain
\begin{equation*}
\RNderivative{A^\dt\rho}{\nu}(t,a,b)
=\RNderivative{A^\dt\rho_t}{(\hd^2\restr\partial\domDelta)\otimes(\hd^2\restr\partial\domDelta)}(a,b)
\leq C\norm{\rho_t}
\leq\frac CT\norm\rho.
\end{equation*}
Since $\RNderivative{A^\sct\rho}{\nu}(t,a,b)=\rho_t(\dom)/[\hd^2(\partial\domDelta)]^2=\norm{\rho_t}/[\hd^2(\partial\domDelta)]^2$ by definition of $A^\sct$,
we obtain
\begin{equation*}
\frac{qp^\sct}{T[\hd^2(\partial\domDelta)]^2}\norm\rho
\leq\RNderivative{\Aq{q}\rho}{\nu}
\leq C\norm{\rho_t}
=\frac{C}{T}\norm{\rho}
\end{equation*}
as desired.
For the discrete setting the result follows from 
$A^\sct\rho(\tau_i\times\Gamma_j\times\Gamma_k)=\frac{\mathcal L(\tau_i)\hd^2(\Gamma_j)\hd^2(\Gamma_k)}{[\hd^2(\partial\domDelta)]^2}\norm\rho$
and $A^\dt\rho(\tau_i\times\Gamma_j\times\Gamma_k)\leq\norm{A^\dt\rho}=\norm\rho$
(where $A^\dt$ and $A^\sct$ still refer to the nondiscrete forward operators).

\end{proof}

For the existence of minimizers we furthermore require the following continuity result.

\begin{lemma}[Continuity of forward operator]\label{LemmaContinuityForwardOp}
	Consider a sequence $(\rho^n,\momentum^n)\in\measp(\spacetime)\times\meas(\spacetime)^3$, $n\in\N$, with uniformly bounded $\BBEnergy(\rho^n,\momentum^n)$.
	If $(\rho^n,\momentum^n)\weakstarto(\rho,\momentum)$ as $n\to\infty$, then
	in the continuous setting, $\RNderivative{\Aq{q}\rho^n}{\nu}$ and $\RNderivative{\Aq{q}\rho}{\nu}$ are uniformly H\"older continuous with exponent $\frac12$.

		Furthermore, $\RNderivative{\Aq{q}\rho^n}{\nu}\to\RNderivative{\Aq{q}\rho}{\nu}$ uniformly 
		in the continuous and discrete setting
		(in the latter even for vanishing positron range).
	\end{lemma}
	\begin{proof}
		Due to the boundedness of $\BBEnergy(\rho^n,\momentum^n)$ the continuity equation \eqref{eqn:continuity} is satisfied,
		and $(\rho^n,\momentum^n)=\d t\otimes(\rho^n_t,\momentum^n_t)$ with $\norm{\rho^n_t}=\frac1T\norm{\rho^n}$ for all $t\in[0,T]$ by \cite[Lemma 1.1.2]{chizat_unbalancedOT}.
		Due to the weak-* convergence, also $(\rho,\momentum)$ satisfy \eqref{eqn:continuity} and admit the analogous disintegration.
		Moreover, the nonnegativity of $\rho^n$ implies $\norm\rho=\lim_{n\to\infty}\norm{\rho^n}$.

			In the continuous setting we have $\RNderivative{\Aq{q}\rho^n}{\nu}=qp^\sct\frac{\norm{\rho^n_t}}{\hd^2(\partial\domDelta)^2}+p^\dt\RNderivative{A^\dt\rho^n}{\nu}$
			so that it suffices to prove the result for $\Aq{q}$ replaced with $A^\dt$.
			The analogous argument holds in the discrete setting.
			
			We first consider the continuous setting.
			Consider an arbitrary subsequence $(\rho^n,\momentum^n)$, still indexed by $n$.
			With a slight abuse of notation we denote the map $(x,t)\mapsto (G\ast\rho^n_t)(x)$ by $G\ast\rho^n$.
			We first show that $G\ast\rho^n$ is H\"older continuous, uniformly in $n$.
			To this end note that for any $\alpha\in C^1_c((0,T))$ the continuity equation \eqref{eqn:continuity} implies
			\begin{multline*}
			\int_0^T\partial_t\alpha(t)\,(G\ast\rho^n_t)(x)\,\d t = 
			\int_0^T\int_\dom\partial_t\alpha(t)\,G(x-y)\,\d\rho^n_{t}(y)\,\d t\\
			=-\int_0^T\int_\dom\alpha(t)\,\nabla G(x-y)\cdot{\d \momentum^n_{t}}(y)\,\d t
			=-\int_0^T\alpha(t)\,(\nabla G\ast\momentum^n_{t})(x)\,\d t,
			\end{multline*}
			thus $t\mapsto-(\nabla G\ast\momentum^n_t)(x)$ is the weak derivative of $t\mapsto(G\ast\rho^n_t)(x)$ for all $x\in\domDeltaHalf$.
			For $r,s\in[0,T]$ and $x,y\in\domDeltaHalf$, using the triangle inequality and twice H\"older's inequality we now obtain the estimate
			\begin{align*}
			\absNorm{(G\ast\rho^n_r)(x)-(G\ast\rho^n_s)(y)}
			&\leq\absNorm{(G\ast\rho^n_r)(x)-(G\ast\rho^n_s)(x)}+\absNorm{(G\ast\rho^n_s)(x)-(G\ast\rho^n_s)(y)}\\
			&=\absNorm{\int_r^s(\nabla G\ast\momentum^n_t)(x)\,\d t}+\absNorm{\int_\dom G(x-z)-G(y-z)\,\d\rho^n_s(z)}\\
			&\leq\|\nabla G\|_{L^\infty}\absNorm{\int_r^s\norm{\rho^n_{t}}^{\frac{1}{2}}\left(\int_\dom\absNorm{\frac{\d \momentum^n_{t}}{\d \rho^n_{t}}}^2\d\rho^n_{t}\right)^{\frac{1}{2}}\d t}
			+C_G\absNorm{x-y}^{\frac12}\norm{\rho^n_s}\\
			&\leq\|\nabla G\|_{L^\infty}\left(\tfrac{\norm{\rho^n}}T\right)^{\frac12}\BBEnergy(\rho^n,\momentum^n)^{\frac{1}{2}}\absNorm{r-s}^{\frac{1}{2}}
			+C_G\absNorm{x-y}^{\frac12}\tfrac{\norm{\rho^n}}T\\
			&\leq C\absNorm{(x,r)-(y,s)}^{\frac12},
			\end{align*}
			where $C_G$ is the H\"older constant of $G$ for exponent $\frac12$, $C$ is a constant depending on $G$ and the bounds on $\norm{\rho^n}$ and $\BBEnergy(\rho^n,\momentum^n)$.
			Thus, by the Arzel\`a--Ascoli theorem there exists a subsequence, still indexed by $n$, such that $G\ast\rho^n$ converges uniformly to some H\"older continuous limit function,
			and by the weak-* convergence $\d t\otimes\rho^n_t\weakstarto\d t\otimes\rho_t$ the limit must be $G\ast\rho$.
			Now since $P$ is easily seen to be continuous from $C^{0,\frac12}(\dom)$ to $C^{0,\frac12}(\mathcal C)$,
			\cref{LemmaForwardDensitySplitting} implies the desired uniform convergence $\RNderivative{A^\dt\rho^n}\nu\to\RNderivative{A^\dt\rho}\nu$.
			
			Now consider the discrete setting without positron range, that is, the operator $B_{\mathrm{pr}}$ is just the identity
			(with positron range the desired uniform convergence is a direct consequence of the result for the continuous setting).
			We first note that $A^\dt\rho^n\weakstarto A^\dt\rho$ by definition of $A^\dt$.
			Now fix some time interval $\tau_i$ and two photon detectors $\Gamma_j,\Gamma_k$. We will show that
			\begin{equation*}
			\partial(\tau_i\times\Gamma_j\times\Gamma_k)\subset(\tau_i\times\partial\Gamma_j\times\partial\Gamma_k)\cup(\partial\tau_i\times\overline{\Gamma_j}\times\overline{\Gamma_k})
			\end{equation*}
			is an $A^\dt\rho$-nullset.
			In fact, due to the disintegration $\rho=\d t\otimes\rho_t$ it suffices to show that $\partial\Gamma_j\times\partial\Gamma_k$ is a $A^\dt\rho_t$-nullset for almost all $t\in\tau_i$.
			Furthermore, since $\meas(\dom)\ni\rho_t\mapsto A^\dt\rho_t(\partial\Gamma_j\times\partial\Gamma_k)\in\R$ is a linear functional,
			within the ball $B=\{\lambda\in\meas(\dom)\,|\,\norm\lambda\leq\norm\rho/T\}$ it takes its extremal values in the extreme points of $B$, which are known to be Dirac masses.
			Therefore it suffices to assume that $\rho_t$ is a Dirac mass, say in $x\in\dom$.
			However, by definition of $A^\dt$ we have
			\begin{align*}
			(A^\dt\delta_x)(\partial\Gamma_j\times\partial\Gamma_k)
			&=(\delta_x\otimes\vol_{S^2})(\{(y,v)\in\dom\times S^2\,|\,R(y,v)\in\partial\Gamma_j\times\partial\Gamma_k\})\\
			&=\vol_{S^2}(\{v\in S^2\,|\,R(x,v)\in\partial\Gamma_j\times\partial\Gamma_k\})\\
			&\leq\vol_{S^2}(\{v\in S^2\,|\,v\text{ is spanned by }y-x\text{ for some }y\in\partial\Gamma_j\}).
			\end{align*}	
			Since $\partial\Gamma_j$ is a $\hd^2$-nullset, then the right-hand side is indeed zero.
			Thus, $\partial(\tau_i\times\Gamma_j\times\Gamma_k)$ is indeed a $A^\dt\rho$-nullset,
			and by the Portmanteau theorem \cite[Thm.\,13.16]{KlenkeProbabilityTheory} we thus have $A^\dt\rho^n(\tau_i\times\Gamma_j\times\Gamma_k)\to A^\dt\rho(\tau_i\times\Gamma_j\times\Gamma_k)$.
			Therefore, $\RNderivative{A^\dt\rho^n}\nu$ in the discrete setting converges pointwise, and since $\nu$ has finite support also uniformly.

		\end{proof}

The final preparation is to show coercivity of our energy functional and boundedness of the minimum.

\begin{lemma}[Coercivity]\label{thm:coercivity}
Let $\beta,p^\dt,p^\sct,q>0$ and $|E|<\infty$,
then there exists a constant $C>0$ such that $\norm\rho,\norm\momentum\leq C[1+J(\rho,\momentum)]$
for $J$ being $J^{E,q}$, $\bar J^{E,q}$, or $\hat J^E$.
\end{lemma}
\begin{proof}
Due to $\hat J^E=J^{E,1}$ and $\bar J^{E,q}\leq J^{E,q}+|E|\log 2$ it suffices to consider $J=J^{E,q}$.

We may assume $J^{E,q}(\rho,\momentum)<\infty$, else there is nothing to show.
Thus, in particular, we may assume $\rho\geq0$, and $(\rho,\momentum)$ satisfy the continuity equation \eqref{eqn:continuity} so that by \cite[Lemma 1.1.2]{chizat_unbalancedOT}
\begin{equation*}
\rho=\d t\otimes\rho_t
\qquad\text{with }
\norm{\rho_t}=\rho_t(\dom)=\frac1T\rho(\spacetime)=\frac1T\norm{\rho}
\text{ for almost all }t\in[0,T].
\end{equation*}
By the pointwise boundedness $\RNderivative{\Aq{q}\rho_n}{\nu}\leq c\norm{\rho}$ in the continuous and discrete setting due to \cref{LemmmaBoundednessForwardOp} we get
\begin{align*}
J^{E,q}(\rho,\momentum)
&\geq(p^\dt+p^\sct)\norm{\rho}-|E|\log(c\norm{\rho})+\beta\BBEnergy(\rho,\momentum)\\
&=(p^\dt+p^\sct)\norm{\rho}-|E|\log\frac{2c|E|}{p^\dt+p^\sct}-|E|\log\frac{(p^\dt+p^\sct)\norm{\rho}}{2|E|}+\beta\BBEnergy(\rho,\momentum)\\
&\geq\frac{p^\dt+p^\sct}2\norm{\rho}-|E|\log\frac{2c|E|}{p^\dt+p^\sct}+\beta\BBEnergy(\rho,\momentum),
\end{align*}
which proves the claim for $\norm{\rho}$.
The result for $\norm\momentum$ then follows from
\begin{equation*}
\norm{\momentum}
=\int_{\spacetime}\left|\RNderivative{\momentum}{\rho}\right|\,\d\rho
\le\sqrt{\rho(\spacetime)}\left(\int_{\spacetime}\left|\RNderivative{\momentum}{\rho}\right|^2\,\d\rho\right)^{\frac12}
=\norm{\rho}^\frac12\BBEnergy(\rho,\momentum)^\frac12
\end{equation*}
and $\beta\BBEnergy(\rho,\momentum)\leq|E|\log\frac{2c|E|}{p^\dt+p^\sct}+J^{E,q}(\rho,\momentum)$.
\end{proof}

\begin{lemma}[Bound on infimum]\label{thm:boundednessInfimum}
Let $\beta,p^\dt,p^\sct,q>0$. There exists $C>0$ such that $\inf J^{E,q}<C(1+|E|)$.
\end{lemma}
\begin{proof}
Set $\rho=\mathcal L^4\restr(\spacetime)$, $\momentum=0$, then obviously
\begin{equation*}
\inf J^{E,q}
\leq J^{E,q}(\rho,\momentum)
\leq\norm{\Aq{q}\rho}+\max_{(t,x,y)\in[0,T]\times\boundDomDelta}\log\left(\RNderivative{(\Aq{q}\rho)}{\nu}(t,x,y)\right)|E|.
\qedhere
\end{equation*}
\end{proof}

With this preparation we can now show existence of minimizers.
We will use that on the space of measures, since it is the dual of a separable Banach space,
the (relative) weak-* topology on normbounded subsets is metrizable so that compactness and sequential compactness coincide on these subsets
and therefore also on the whole space.
Thus sequential weak-* compactness and weak-* compactness coincide.

\begin{theorem}[Existence of reconstructions]\label{thm:ExistenceMinimizers}
Let $\beta,p^\dt,p^\sct,q>0$ and let the measurement $E$ be a realization of $\bm E=\Poi{A\rho^\dagger}$
for some ground truth material distribution $\rho^\dagger=\dt t\otimes\rho_t^\dagger\in\measp(\spacetime)$.
Then almost surely (in particular for realizations $E$ of $\bm E$ with $|E|<\infty$) the set of minimizers of $J^{\bm E,q}$ is non-empty and compact with respect to weak-* convergence.
\end{theorem}
\begin{proof}
Since $|E|$ is Poisson distributed with parameter
$A\rho^\dagger([0,T]\times\boundDomDelta)=\norm{A\rho^\dagger}=p^\dt\norm{A^\dt\rho^\dagger}+p^\sct\norm{A^\sct\rho^\dagger}=(p^\dt+p^\sct)\norm{\rho^\dagger}<\infty$
we have $|E|<\infty$ almost surely.
From now on let $E$ be such a realization of $\bm E$.

We show existence of a minimizer by the direct method of the calculus of variations.
From \cref{thm:coercivity} we directly see that $J^{E,q}$ is bounded from below.
Now consider a minimizing sequence $(\rho^n,\momentum^n)\subset\measp(\spacetime)\times\meas(\spacetime)^3$, $n\in\N$,
such that $J^{E,q}(\rho^n,\momentum^n)\to\inf J^{E,q}$ monotonically as $n\to\infty$.
Without loss of generality we may assume $J^{E,q}(\rho^n,\momentum^n)\leq C<\infty$ by \cref{thm:boundednessInfimum}.
Thus, in particular, we may assume $\rho^n\geq0$ for all $n$, and $(\rho^n,\momentum^n)$ satisfy the continuity equation \eqref{eqn:continuity} so that by \cite[Lemma 1.1.2]{chizat_unbalancedOT} $\rho^n=\d t\otimes\rho_t^n$.
By \cref{thm:coercivity} we have uniform boundedness of $\norm{\rho^n}$ and $\norm{\momentum^n}$
and therefore weak-* convergence along a subsequence which we still denote by $(\rho^n,\momentum^n)$, that is,
$(\rho^n,\momentum^n)\weakstarto(\rho,\momentum)\in\measp\times\meas^3$.
Using the weak-* continuity of the forward operator established in \cref{LemmaContinuityForwardOp} and the weak-* lower semi-continuity of $\BBEnergy$ \cite[Thm.\,5.18]{OTAppliedMath} we arrive at
\begin{multline*}
\inf J^{E,q}
=\liminf_{n\to\infty} J^{E,q}(\rho^n,\momentum^n)
=\liminf_{n\to\infty}\left(\norm{A\rho^n}-\int\log\left(\RNderivative{(\Aq{q}\rho^n)}{\nu}\right)\d E+\beta\BBEnergy(\rho^n,\momentum^n)\right)\\
\ge\norm{A\rho}-\int\log\left(\RNderivative{(\Aq{q}\rho)}{\nu}\right)\d E+\beta\BBEnergy(\rho,\momentum)
=J^{E,q}(\rho,\momentum),
\end{multline*}
which means that $(\rho,\momentum)$ is a minimizer of $J^{E,q}$.

Since for any minimizer $(\hat\rho,\hat\momentum)$ it holds $J^{E,q}(\hat\rho,\hat\momentum) \le J^{E,q}(\mathcal{L}\restr[0,T]\times\dom,0)$ we can repeat the above argument to deduce the sequential weak-* compactness of the set of minimizers.
\end{proof}

We close the section with a continuity result.

\begin{proposition}[Minimizers of sequences of measurements]\label{thm:GammaConvergence}
Let $\beta,p^\dt,p^\sct,q>0$, let the measurements $E_n$, $n\in\N$, converge weakly-* to some $E\in\measp([0,T]\times\boundDomDelta)$,
and let $(\rho_n,\momentum_n)$ be minimizers of $J^{E_n,q}$.
Then $(\rho_n,\momentum_n)$ contains a subsequence converging to a minimizer of $J^{E,q}$.
\end{proposition}
\begin{proof}
By the weak-* convergence, $|E_n|$ is uniformly bounded, thus by \cref{thm:coercivity,thm:boundednessInfimum} we have uniform boundedness of $\norm{\rho_n},\norm{\momentum_n}$.
As a consequence there exists a weakly-* convergent subsequence, for simplicity again denoted $(\rho_n,\momentum_n)\weakstarto(\rho,\momentum)$.
Now let $(\tilde\rho,\tilde\momentum)$ be any competitor to $(\rho,\momentum)$, then
\begin{equation*}
J^{E,q}(\rho,\momentum)
\leq\lim_{n\to\infty}J^{E_n,q}(\rho_n,\momentum_n)
\leq\lim_{n\to\infty}J^{E_n,q}(\tilde\rho,\tilde\momentum)
=J^{E,q}(\tilde\rho,\tilde\momentum)
\end{equation*}
where the first inequality follows from the weak-* convergence of $\rho_n,\momentum_n$ and \cref{LemmaContinuityForwardOp}
and the last equality from the continuity of $\Aq{q}\tilde\rho$.
\end{proof}

\subsection{Scaling behaviour of reconstruction functional}\label{sec:scaling}
So far we simply ignored the radionuclide halflife for notational simplicity.
Reintroducing it we would get $\bm E=\Poi{\frac{\ln2}{\halflife}A\rho^\dagger}$ for the ground truth radionuclide distribution $\rho^\dagger$
as well as the reconstruction functional
\begin{align*}
J^{\bm E,q,\beta,\halflife,T,\dom}(\rho,\momentum)&=\frac{\ln2}{\halflife}\norm{A\rho}-\int\Log{\frac{\ln2}{\halflife}\RNderivative{\Aq{q}\rho}{\nu}}\d \bm E+\beta\BBEnergy(\rho,\momentum).
\end{align*}
This reconstruction functional contains the regularization parameter $\beta$ which has to be chosen before reconstruction
(recall that a heuristic for the choice of the debiasing parameter $q$ was given in \cref{rem:heuristicDebiasing}).
Its choice of course depends on the system parameters such as the radionuclide halflife, the typical spatial length scale, or the total amount of radionuclide.
To identify this dependence we need to perform a nondimensionalization.
To this end let us introduce the spatiotemporal rescaling
\begin{equation*}
\bar s_{\theta,\lambda}:\R\times\R^3\to\R\times\R^3,\quad
(t,x)\mapsto(\tfrac t\theta,\tfrac x\lambda)
\end{equation*}
with temporal scale $\theta$ and length scale $\lambda$ as well as
\begin{equation*}
\bar S_{\theta,\lambda}:\meas(\R\times\R^3)^k\to\meas(\R\times\R^3)^k,\quad
\mu\mapsto\tfrac1\theta\pushforward{(\bar s_{\theta,\lambda})}{\mu}
\end{equation*}
for any $k\in\N$.
The operator $\bar S_{\theta,\lambda}$ is the natural representation of the spatiotemporal coordinate change $(t,x)\mapsto(t/\theta,x/\lambda)$:
This coordinate change turns a time-dependent mass distribution $\rho_t\in\measp(\R^3)$ into $\pushforward{s_\lambda}\rho_{t/\theta}$ for $s_\lambda(x)=x/\lambda$,
and therefore $\bar S_{\theta,\lambda}$ is chosen to satisfy
\begin{equation*}
\bar S_{\theta,\lambda}(\d t\otimes\rho_t)
=\d t\otimes\pushforward{(s_\lambda)}\rho_{t\theta}.
\end{equation*}
On the other hand, the coordinate change turns an observed measurement $E$ into $S_{\theta,\lambda}E$ for
\begin{equation*}
S_{\theta,\lambda}:\meas(\R\times\R^3\times\R^3)\to\meas(\R\times\R^3\times\R^3),\quad
E\mapsto\pushforward{(s_{\theta,\lambda})}E
\qquad\text{with }
s_{\theta,\lambda}(t,x,y)=(\tfrac t\theta,\tfrac x\lambda,\tfrac y\lambda).
\end{equation*}
The following nondimensionalization is now straightforward.

\begin{lemma}[Scaling invariances of reconstruction functional]\label{thm:scaleInvariance}
Given a time, length, and mass scale $\theta,\lambda,\mu>0$ and $E\in\measp([0,T]\times\boundDomDelta)$ we set
\begin{equation*}
\hat\beta=\beta\mu\frac{\lambda^2}{\theta},\qquad
\halflifeHat=\frac\halflife{\mu\theta},\qquad
\hat T=\frac T\theta,\qquad
\hat\dom=\frac1\lambda\dom,\qquad
\hat E=S_{\theta,\lambda}E.
\end{equation*}
The measures $(\rho,\momentum)\in\measp([0,T]\times\dom)\times\meas([0,T]\times\dom)^3$ minimize $J^{E,q,\beta,\halflife,T,\dom}$
if and only if $(\hat\rho,\hat\momentum)=\bar S_{\theta,\lambda}(\frac1\mu\rho,\frac\theta{\mu\lambda}\momentum)$ minimize $J^{\hat E,q,\hat\beta,\halflifeHat,\hat T,\hat\dom}$.
\end{lemma}
\begin{proof}
This immediately follows from the straightforward identity
$J^{E,q,\beta,\halflife,T,\dom}(\rho,\momentum)=J^{\hat E,q,\hat\beta,\halflifeHat,\hat T,\hat\dom}(\hat\rho,\hat\momentum)+R$,
where the remainder $R$ is independent of $(\hat\rho,\hat\momentum)$.
\end{proof}

In other words, if two PET scans are conducted whose measurements are coincidentally related by the simple spacetime rescaling $s_{\theta,\lambda}$
and if the halflifes of the employed radionuclides are $\halflife$ and $\halflifeHat$, respectively,
then by choosing the regularization parameters $\beta$ and $\hat\beta=\beta\frac\halflifeHat\halflife\frac{\lambda^2}{\theta^2}$, respectively,
one obtains the same reconstruction up to a simple mass and spacetime rescaling.
Of course, the question naturally arises under what circumstances the measurements of two experiments are related by a simple spacetime rescaling, at least in law.
This question is answered in the next statement,
for which we decorate the random variable $\boldsymbol E$ with the parameters it depends on, in particular the groundtruth mass distribution $\rho^\dagger$,
\begin{equation*}
\boldsymbol E^{\rho^\dagger,\halflife,T,\dom}=\Poi{\tfrac{\ln2}\halflife A\rho^\dagger}.
\end{equation*}

\begin{proposition}[Scaling invariance of measurement]\label{thm:invarianceMeasurement}
Given a time, length, and mass scale $\theta,\lambda,\mu>0$ we set
\begin{equation*}
\hat\rho^\dagger=\frac1\mu\bar S_{\theta,\lambda}\rho^\dagger,\qquad
\halflifeHat=\frac\halflife{\mu\theta},\qquad
\hat T=\frac T\theta,\qquad
\hat\dom=\frac1\lambda\dom.
\end{equation*}
Then the law of $\boldsymbol E^{\hat\rho^\dagger,\halflifeHat,\hat T,\hat\dom}$
equals the pushforward of the law of $\boldsymbol E^{\rho^\dagger,\halflife,T,\dom}$ under the map $S_{\theta,\lambda}$.
\end{proposition}
\begin{proof}
By the mapping theorem \cite[Thm.\,5.1]{bookPPP_last_penrose},
given a Poisson point process $\boldsymbol E=\Poi{\lambda}$ with intensity $\lambda$,
any composition $F\circ\boldsymbol E$ with the pushforward $F=\pushforward f{}$ under a measurable map $f$ is also a Poisson point process and has intensity $\pushforward f\lambda$.
Therefore, $S_{\theta,\lambda}\circ\boldsymbol E^{\rho^\dagger,\halflife,T,\dom}$ is a Poisson point process with intensity
\begin{equation*}
\pushforward{(s_{\theta,\lambda})}{(\tfrac{\ln2}\halflife A\rho^\dagger)}
=\tfrac{\ln2}{\halflifeHat} A(\hat\rho^\dagger),
\end{equation*}
thus $S_{\theta,\lambda}\circ\boldsymbol E^{\rho^\dagger,\halflife,T,\dom}=\boldsymbol E^{\hat\rho^\dagger,\halflifeHat,\hat T,\hat\dom}$.
(Note that the forward operator $A$ on the left-hand side is implicitly understood as the one for the domain $\dom$,
while $A$ on the right-hand side is the forward operator for domain $\hat\dom$.)
\end{proof}

As a consequence of the previous statements we can determine when our reconstructions will be related (in law) by simple rescalings.
To this end let us introduce the (set-valued) reconstruction mapping
\begin{equation*}
\reconMap^{q,\beta,\halflife,T,\dom}:
E\mapsto\begin{cases}\argmin_{(\rho,\momentum)\in}J^{E,q,\beta,\halflife,T,\dom}(\rho,\momentum)&\text{if }|E|<\infty,\\
\emptyset&\text{else,}
\end{cases}
\end{equation*}
which is well-defined by \cref{thm:ExistenceMinimizers}.
It is set-valued due to the potential nonuniqueness of minimizers,
and for measurements with $|E|=\infty$ (which occur with zero probability and for which a minimizer might not exist) we simply set it to the empty set.
In \cref{sec:measurability} we prove the map $\reconMap^{q,\beta,\halflife,T,\dom}$ to be measurable (the corresponding measurable spaces are also specified in \cref{sec:measurability}).
Therefore, the (stochastic) reconstruction
\begin{equation*}
\boldsymbol\recon^{\rho^\dagger,q,\beta,\halflife,T,\dom}
=\reconMap^{q,\beta,\halflife,T,\dom}\circ\boldsymbol E^{\rho^\dagger,\halflife,T,\dom}
\end{equation*}
is a random variable, and we obtain the following.

\begin{corollary}[Scaling invariance of reconstruction]
Using the notation from \cref{thm:scaleInvariance,thm:invarianceMeasurement},
the law of $\boldsymbol\recon^{\hat\rho^\dagger,q,\hat\beta,\halflifeHat,\hat T,\hat\dom}$ equals the pushforward
of the law of $\boldsymbol\recon^{\rho^\dagger,q,\beta,\halflife,T,\dom}$ under the map $\bar S_{\theta,\lambda}$.
\end{corollary}
\begin{proof}
By \cref{thm:scaleInvariance} we have $\bar S_{\theta,\lambda}\circ\reconMap^{q,\beta,\halflife,T,\dom}=\reconMap^{q,\hat\beta,\halflifeHat,\hat T,\hat\dom}\circ S_{\theta,\lambda}$.
The result now follows from \cref{thm:invarianceMeasurement} via
\begin{multline*}
\boldsymbol\recon^{\hat\rho^\dagger,q,\hat\beta,\halflifeHat,\hat T,\hat\dom}
=\reconMap^{q,\hat\beta,\halflifeHat,\hat T,\hat\dom}\circ\boldsymbol E^{\hat\rho^\dagger,\halflifeHat,\hat T,\hat\dom}
=\reconMap^{q,\hat\beta,\halflifeHat,\hat T,\hat\dom}\circ S_{\theta,\lambda}\circ\boldsymbol E^{\rho^\dagger,\halflife,T,\dom}\\
=\bar S_{\theta,\lambda}\circ\reconMap^{q,\beta,\halflife,T,\dom}\circ\boldsymbol E^{\rho^\dagger,\halflife,T,\dom}
=\bar S_{\theta,\lambda}\circ\boldsymbol\recon^{\rho^\dagger,q,\beta,\halflife,T,\dom}.
\qedhere
\end{multline*}
\end{proof}

Let us now return to the question of finding a good regularization parameter $\beta$.
Due to the above scale invariance it suffices to restrict to the situation
in which the total mass $\|\rho^\dagger_t\|$ for all times $t$, the spatial scale $l$ of $\rho^\dagger$ (for instance the typical bending radius of the particle trajectories), and the typical particle velocity $v$ of $\rho^\dagger$ are all equal to one.
In other words, we pick $\theta=l/v$, $\lambda=l$, and $\mu=\|\rho^\dagger_t\|$ and seek the optimal $\hat\beta$ for the resulting $\halflifeHat$, $\hat T$, and $\hat\dom$.
The original $\beta$ is then obtained as $\beta=\hat\beta\frac{\theta}{\mu\lambda^2}$.
The optimal $\hat\beta$ is that for which the reconstruction $\hat\rho$ deviates the least from the ground truth $\hat\rho^\dagger$,
averaged over all $\hat\rho^\dagger$ with unit mass, unit spatial, and unit velocity scale as well as over all associated measurements.
This $\hat\beta$ will have to be determined experimentally based on realistic samples of $\hat\rho^\dagger$.
In principle $\hat\beta$ may depend on $\halflifeHat$, $\hat T$, and $\hat\dom$.
However, the dependence on $\hat T$ is expected to be negligible:
The reconstruction from measurements on a long time interval should behave roughly the same
as when the long time interval is split into shorter ones on which one performs separate reconstructions.
A similar consideration suggests that $\hat\beta$ only weakly depends on $\hat\dom$.
Thus we expect the optimal $\hat\beta$ to be a function $\hat\beta(\halflifeHat)$ so that one should pick $\beta=\hat\beta(\halflifeHat)\frac{\theta}{\mu\lambda^2}$ or equivalently
\begin{equation*}
\beta=\frac{\hat\beta\left(\frac{v\halflife}{l\|\rho^\dagger_t\|}\right)}{\halflife v^2}
\qquad\text{for some function }\hat\beta,
\end{equation*}
where $v$ is the typical particle velocity, $l$ the typical spatial scale of the particle trajectories, and $\|\rho^\dagger_t\|$ the used amount of radionuclide.

To close the section, let us briefly illustrate the invariances for $\theta,\lambda,\mu>0$:
\begin{enumerate}
\item
If $\hat\rho^\dagger=\rho^\dagger/\mu$,
$\halflifeHat=\halflife/\mu$,
$\hat T=T$,
$\hat\dom=\dom$,
$\hat\beta=\beta\mu$,
then in law the reconstruction satisfies $\hat\rho=\rho/\mu$.
\item
If $\hat\rho^\dagger_t=\rho^\dagger_{\theta t}$,
$\halflifeHat=\halflife/\theta$,
$\hat T=T/\theta$,
$\hat\dom=\dom$,
$\hat\beta=\beta/\theta$,
then in law the reconstruction satisfies $\hat\rho_t=\rho_{\theta t}$.
\item
If $\hat\rho^\dagger_t(A)=\rho^\dagger_t(\lambda A)$ for all measurable $A\subset\R^3$,
$\halflifeHat=\halflife$,
$\hat T=T$,
$\hat\dom=\frac1\lambda\dom$,
$\hat\beta=\beta\lambda^2$,
then in law the reconstruction satisfies $\hat\rho_t(A)=\rho_{\theta t}(\lambda A)$ for all measurable $A\subset\R^3$.
\end{enumerate}

\subsection{Structure of minimizers}\label{sec:minimizerStructure}
Our reconstruction functional essentially consists of a convex term penalizing the deviation from a finite measurement $E$ and a convex regularization.
Via so-called representer theorems it can often be shown that functionals of such type have minimizers
that are finite linear combinations (where the number depends on the dimension of the measurement) of extreme points of the $1$-sublevel set of the regularization.
A particular instance of this phenomenon for the setting of Benamou--Brenier optimal transport regularization is the following.

\begin{theorem}[{\cite[Thm.\,10]{Bredies_extremePointsBB}}]\label{Thm_Bredies_characterization}
Let $\dom\subset\R^d$ be the closure of an open bounded domain,
$H$ a finite-dimensional Hilbert space,
and denote with $C_W([0,1];\meas(\dom))$ the family of narrowly continuous curves in $\meas(\dom)$.
Consider the functional $J:C_W([0,1];\meas(\dom))\times\meas(\spacetime)^d\to\R\cup\{\infty\}$,
\begin{align*}
J(\rho,\momentum)=F(B\rho)+\alpha\norm{\rho}+\beta\BBEnergy(\rho,\momentum),
\end{align*}
where $F:H\to\R\cup\{\infty\}$ is convex, lower semi-continuous, and bounded from below,
$B:C_W([0,1];\meas(\dom))\to H$ is linear and continuous in the sense that $\rho^n_t\weakstarto\rho_t$ in $\meas(\dom)$ for every $t$ implies $B\rho^n\to B\rho$ in $H$,
and $J$ is proper.
Then for any $\alpha,\beta>0$ there exists a minimizer $(\hat\rho,\hat\momentum)$ of $J$ of the form
\begin{align*}
(\hat\rho,\hat\momentum) = \sum_{i=1}^nc_i(\rho^i,\momentum^i),
\qquad\text{with}\quad
c_i>0,\qquad
\rho^i = \d t\otimes\delta_{\gamma_i(t)}, \qquad
\momentum^i=\dot\gamma_i\rho^i,\qquad
i=1,\ldots,n,
\end{align*}
where $K\le\mathrm{dim}(H)<\infty$ and
$\gamma_i\in H^1((0,1))^d$, the Sobolev space of curves in $\R^d$ with square-integrable weak derivative, with $\gamma_i(t)\in\dom$ for each $t\in[0,1]$.
\end{theorem}

In other words, the minimizer is a finite linear combination of travelling Dirac masses.
(\cite[Thm.\,10]{Bredies_extremePointsBB} in addition states that the regularization $\alpha\norm{\hat\rho}+\beta\BBEnergy(\hat\rho,\hat\momentum)$ of the minimizer
equals the sum of the regularization values of its single components $(\rho^i,\momentum^i)$, which readily follows from its structure.)
This result can immediately be applied to our reconstruction functional to show that also our reconstructions will be finite linear combinations of travelling Dirac masses.

\begin{corollary}[Structure of reconstructions]\label{thm:minimizerStructure}
Let $\beta,p^\dt,p^\sct,q>0$ and $|E|<\infty$,
then $J^{E,q}$ has a minimizer $(\hat\rho,\hat\momentum)$ of the form
\begin{align*}
(\hat\rho,\hat\momentum) = \sum_{i=1}^nc_i(\rho^i,\momentum^i),
\qquad\text{with}\quad
c_i>0,\qquad
\rho^i = \d t\otimes\delta_{\gamma_i(t)}, \qquad
\momentum^i=\dot\gamma_i\rho^i,\qquad
i=1,\ldots,n,
\end{align*}
where $n\le|E|$ and $\gamma_i\in H^1((0,T))^d$ with $\gamma_i(t)\in\dom$ for each $t\in[0,T]$.
\end{corollary}
\begin{proof}
By \cref{thm:scaleInvariance} it suffices to consider the case $T=1$.
Furthermore note that $\norm{A\rho}=(p^\sct+p^\dt)\norm{\rho}$ so that
\begin{multline*}
J^{E,q}(\rho,\momentum)=F(B\rho)+\alpha\norm\rho+\beta\BBEnergy(\rho,\momentum)\\
\quad\text{for }
\alpha=p^\sct+p^\dt,\quad
B\rho=\left(\RNderivative{\Aq{q}\rho}{\nu}(t,a,b)\right)_{(t,a,b)\in E}\in\R^{|E|},\quad
F(v)=-\sum_{k=1}^{|E|}\log(v_i).
\end{multline*}
By \cref{LemmaContinuityForwardOp}, the operator $B$ satisfies the condition of \cref{Thm_Bredies_characterization}.
(Note that in \cref{LemmaContinuityForwardOp} we actually only prove the required continuity of $B$
along sequences $(\rho^n,\momentum^n)$ with uniformly bounded $\BBEnergy(\rho^n,\momentum^n)$, however,
only this is needed in the proof of \cref{Thm_Bredies_characterization}.
Along general sequences the desired continuity of $B$ can readily be derived from \cref{LemmaForwardDensitySplitting}.)
Finally, $F$ satisfies the properties of \cref{Thm_Bredies_characterization} except for the boundedness from below,
however, that property is only used in the proof of \cref{Thm_Bredies_characterization} to show existence of minimizers, which we already have by \cref{thm:ExistenceMinimizers}.
Therefore, the structure of the minimizer follows from \cref{Thm_Bredies_characterization}.
\end{proof}

\subsection{Functional lifting and convex relaxation}\label{sec:lifting}

In \cite{ScScWi20} we already illustrated the close relation of our reconstruction model to an approach from \cite{Lee2015} for tracking a single radioactively labelled cell.
Here we briefly motivate our model as a convex relaxation of a multiparticle tracking version of \cite{Lee2015}.
To this end suppose we already know there are $n$ distinct radiolabelled travelling particles which are to be reconstructed from the PET measurement $E=(t_k,a_k,b_k)_{k=1,\ldots,K}$.
Those particles can be described by their mass $m_i>0$ and their trajectory $\gamma_i:[0,T]\to\dom$, $i=1,\ldots,n$.
The spatiotemporal radioactive material distribution and momentum is then given by
\begin{equation*}
\rho=\sum_{i=1}^nm_i\d t\otimes\delta_{\gamma_i(t)},\qquad
\momentum=\sum_{i=1}^nm_i\d t\otimes\dot\gamma_i(t)\delta_{\gamma_i(t)}.
\end{equation*}
Abbreviating by $L_k=a_k+\R(b_k-a_k)$ the line of response associated with the $k$th photon pair detection,
our functional (for simplicity assuming zero scatter probability $p^\sct=0$) applied to this linear combination of travelling Dirac masses becomes
\begin{equation*}
J^{E,q}(\rho,\momentum)
=p^\dt\sum_{i=1}^nm_i-\sum_{k=1}^K\log(p^\dt(G*\rho)(L_k))+\beta\sum_{i=1}^nm_i\int_0^T|\dot\gamma_i|^2\,\d t.
\end{equation*}
If the positron range kernel is taken as a Gaussian $G(x)=\exp(-x^2/2\sigma^2)/(2\pi)^{3/2}\sigma^3$ of variance $\sigma^2$, this turns into
\begin{align}\label{eqn:discreteModel}
&J^{E,q}(\rho,\momentum)\nonumber\\
=&p^\dt\sum_{i=1}^nm_i-\sum_{k=1}^K\log\left(\sum_{i=1}^nm_i\exp\left(-\frac{\mathrm{dist}(\gamma_i(t_k),L_k)^2}{2\sigma^2}\right)\right)
+\beta\sum_{i=1}^nm_i\int_0^T|\dot\gamma_i|^2\,\d t-K\log(p^\dt/((2\pi)^{3/2}\sigma^3)).
\end{align}
For a single particle, $n=1$, of unit mass $m_1=1$, this becomes
\begin{equation*}
J^{E,q}(\rho,\momentum)
=\sum_{k=1}^K\frac{\mathrm{dist}(\gamma_1(t_k),L_k)^2}{2\sigma^2}
+\beta\int_0^T|\dot\gamma_1|^2\,\d t+R,
\end{equation*}
where the remainder $R$ is independent of the optimization variables $\rho,\momentum$.
Optimizing this functional for the curve $\gamma_1$ is a convex optimization problem, whose cubic spline discretization was considered in \cite{Lee2015}.
However, for $n>1$, the more general functional \eqref{eqn:discreteModel} with $n$ travelling particles is highly nonconvex due to the second term which represents the de facto combinatorial problem of identifying which particle trajectory $\gamma_i$ most probably caused the detection event $(t_k,a_k,b_k)\in E$.
Minimizing \eqref{eqn:discreteModel} yields the MAP estimate among all configurations of $n$ travelling particles
(indeed, our modelling in \cref{sec:reconstructionFunctional} under this constraint would lead to exactly \eqref{eqn:discreteModel}).
Viewing this functional on $n$ paths as a restriction of the functional $J^{E,q}$, which operates on paths of measures, is sometimes called a \emph{functional lifting} into a higher-dimensional space:
Instead of optimizing over three-dimensional particle positions, one then optimizes over the empirical measures $\rho_t$, that is, measures on $\R^3$ that describe the particle configuration.
Then dropping the (nonconvex) constraint that the measures $(\rho,\momentum)$ need to represent exactly $n$ particles with nonchanging mass leads to the convex optimization functional $J^{E,q}$,
so one can view $J^{E,q}$ as a convex relaxation of the model \eqref{eqn:discreteModel}.
For fixed $n$ this relaxation is not tight, since $J^{E,q}$ is oblivious to $n$. So the relaxation cannot be used to restrict reconstruction to a fixed number of particles, as might be motivated by prior knowledge.
However, if $n$ is unknown and no prior knowledge is available, i.e.~if one wants to minimize \eqref{eqn:discreteModel} also in $n$, then the relaxation does indeed become tight as can be seen from \cref{thm:minimizerStructure}, which guarantees the existence of minimizers that are composed from a finite number of discrete particles.

\section*{Acknowledgements}
MM's and BW's work was supported by the Deutsche Forschungsgemeinschaft (DFG, German Research Foundation)
under Germany's Excellence Strategy -- EXC 2044 --, Mathematics M\"unster: Dynamics -- Geometry -- Structure,
and under the Collaborative Research Centre 1450--431460824, InSight, University of M\"unster.\\
BS was supported by the Emmy Noether Programme of the DFG (project
number 403056140)

\begin{appendices}
	\appendix
	\section{Measurability of reconstruction mapping}\label{sec:measurability}
	Here we briefly show the measurability of the map $\reconMap^{q,\beta,\halflife,T,\dom}$ from \cref{sec:scaling}.
	This first requires to properly specify the domain and codomain as measurable spaces.
	Since the domain consists of realizations of a point process and the codomain consists of weakly-* compact subsets of a space of Radon measures,
	we first recapitulate the basics of point processes (following \cite{bookPPP_last_penrose}) and of spaces of sets (following \cite{KuratowskiTopology, InfDimAnalysis_Aliprantis,Michael1951TopologiesOS}).
	
	Let $(Z,\mathcal{Z})$ be a measurable space and abbreviate $\N_0=\{0,1,2,\ldots\}$.
	A point processes on $Z$ can be seen as random countable subsets of $Z$ or equivalently as random $(\N_0\cup\{\infty\})$-valued measures:
	Let $N_{<\infty}$ denote the space of all measures $\mu$ on $Z$ satisfying $\mu(B)\in\N_0$ for all $B\in\mathcal{Z}$,
	and let ${N}$ be the space of all measures that can be written as a countable sum of measures from ${N}_{<\infty}$.
	Let further $\bm{\mathcal{N}}$ denote the $\sigma$-algebra generated by the collection of all subsets of ${N}$ having the form
	\begin{align*}
	\left\{\mu\in{N} \ | \ \mu(B)=k\right\}\quad\text{for some } B\in\mathcal{Z}, k\in\mathbb{N}_0.
	\end{align*}
	Thus ${\mathcal N}$ is the smallest $\sigma$-algebra on ${N}$ such that $\mu \mapsto\mu(B)$ is measurable for all $B\in\mathcal{Z}$.
	Equivalently, ${\mathcal N}$ is generated by the integration maps $\pi_f:\mu\mapsto\int f\d\mu$ for $f:Z\to\R$ a nonnegative measurable function,
	\begin{equation*}
	{\mathcal N}=\sigma\{ \pi_f^{-1}(B) \ | \ f:Z\to\R\text{ nonnegative and measurable, } B\subset\R\text{ measurable} \}.
	\end{equation*}
	A \emph{point process} on $Z$ is an $({N},{\mathcal N})$-valued random variable.
	
	Next we discuss topological spaces of sets (which automatically turn into measurable spaces when equipped with the Borel $\sigma$-algebra).
	Given a topological $T_1$ space $Y$, its so-called hyperspace $2^Y$ is the set of all non-empty closed subsets $C\subset Y$.
	We endow it with the so-called Vietoris or exponential topology,
	the coarsest topology in which the sets $2^A$ are open in $2^Y$ for $A$ open in $Y$ and closed for $A$ closed in $Y$ \cite[Ch.\,17]{KuratowskiTopology},
	where $2^A$ for $A\subset Y$ denotes all subsets of $A$ that are closed in $Y$.
	Note, that the hyperspace $2^Y$ is in general more suitable for an analysis than the power set of all non-empty subsets of $Y$
		as the latter has poor separation properties and fails to be $T_1$ for general topological spaces \cite{Michael1951TopologiesOS}. In more detail: The sets
		\begin{align}\label{eqn:basisExponentialTopology}
		\langle U_1, \ldots, U_n\rangle=\left\{ E\in 2^Y \ \middle| \ E\subset\bigcup_{i=1}^n U_i, \ E\cap U_i\neq\emptyset \right\}
		\end{align}
		for $U_i$ open in $Y$ form a basis and generate the exponential topology \cite[Def.\,1.7, 1.6a]{Michael1951TopologiesOS} (the author calls it finite topology). If one generalizes this to the power set $\mathrm{Power}(Y)$ of all non-empty subsets of $Y$, i.e.\ if one considers the topology generated by
		\begin{align*}
		\langle U_1, \ldots, U_n\rangle^+=\left\{ E\in \mathrm{Power}(Y)\setminus\emptyset \ \middle| \ E\subset\bigcup_{i=1}^n U_i, \ E\cap U_i\neq\emptyset \right\}
		\end{align*}
		for $U_i$ open in $Y$, then this topology fails to be $T_1$ because given a set $E\subset Y$,  any neighbourhood of $\overline{E}$ contains $E$ (if one interprets these sets as elements of the power set).

	We further denote by $\mathcal K(Y)$ the set of all non-empty compact subsets of $Y$.
	If $Y$ is metrizable, then the Hausdorff metric can be defined on $\mathcal K(Y)$,
	and the (relative) exponential topology coincides on $\mathcal K(Y)$ with the Hausdorff metric topology \cite[Thm. 3.91]{InfDimAnalysis_Aliprantis}. If $Y$ is a locally compact $T_1$ space, then $\mathcal K(Y)$ is open in $2^Y$ \cite[Prop.\,4.4]{Michael1951TopologiesOS}.
	Let us furthermore note that for closed $G\subset2^Y$ and closed $C\subset Y$, the set $G\cap2^C$ is closed in the hyperspace $2^C$ with its exponential topology.
	Indeed, since the family of sets $\{2^A\ |\ A\subset Y \text{ open}\}\cup\{2^Y\setminus 2^A\ |\ A\subset Y \text{ closed}\}$ are a subbase for the exponential topology on $2^Y$
	, the set $G$ can be written as
	\begin{align*}
	G=\bigcap_{i\in I_1}\bigcup_{j\in I_2}\left(2^{G_{ij}}\cup 2^{Y}\setminus 2^{\tilde G_{ij}} \right)
	\end{align*}
	for open sets $\tilde G_{ij}$, closed sets $G_{ij}$, an arbitrary index set $I_1$, and a finite index set $I_2$. It follows
	\begin{align*}
	G\cap 2^C
	=\bigcap_{i\in I_1}\bigcup_{j\in I_2}2^C\cap\left(2^{G_{ij}}\cup 2^{Y}\setminus 2^{\tilde G_{ij}} \right)
	= \bigcap_{i\in I_1}\bigcup_{j\in I_2}\left( 2^{G_{ij}\cap C}\cup 2^C\setminus 2^{\tilde G_{ij}\cap C} \right),
	\end{align*}
	which is of the same form as $G$ except that the open and closed sets are now relative to $2^C$. Hence, $G\cap 2^C$ is closed in $2^C$.
	
	Finally we consider general setvalued maps (such as $\reconMap^{q,\beta,\halflife,T,\dom}$) and their measurability properties.
	A multivalued function $\varphi$ from a domain $X$ to a codomain $Y$ assigns to each argument from $X$ a subset of $Y$.
	Using the notation from \cite[\S\,18]{InfDimAnalysis_Aliprantis} we call $\varphi$ a correspondence and write $\varphi:X\twoheadrightarrow Y$. One major difference between functions and correspondences is that for the latter multiple different notions of measurability exist. If $(X,\mathcal X)$ is a measurable space and $Y$ a topological space, then $\varphi$ is called 
	\begin{itemize}
		\item \emph{weakly measurable}, if  $\left\{x\in X \ \middle| \ \varphi(x)\cap O\neq\emptyset \right\}\in\mathcal X$ for all open sets $O\subset Y$;
		\item \emph{measurable}, if  $\left\{x\in X \ \middle| \ \varphi(x)\cap F\neq\emptyset \right\}\in\mathcal X$ for all closed sets $F\subset Y$;
		\item \emph{Borel measurable}, if $\left\{x\in X \ \middle| \ \varphi(x)\cap B\neq\emptyset \right\}\in\mathcal X$ for all Borel subsets $B\subset Y$.
	\end{itemize}

	If $Y$ is separable and metrizable and $\varphi$ maps into $\mathcal K(Y)$, measurability of $\varphi$ can be reduced to weak measurability as follows.
	
	\begin{theorem}[{\cite[Thm.\,18.10]{InfDimAnalysis_Aliprantis}}]\label{ThmBorelMeasCompactSubspaces}
		Let $(X,\mathcal X)$ be a measurable space and $Y$ be a separable metrizable space.
		For $\varphi:X\twoheadrightarrow Y$ with values in $\mathcal K(Y)$ the following statements are equivalent:
		\begin{enumerate}
			\item The correspondence $\varphi$ is weakly measurable;
			\item The correspondence $\varphi$ is measurable;
			\item The correspondence $\varphi$ is Borel measurable as a map $\varphi:(X,\mathcal X)\to\mathcal K(Y)$ with the Hausdorff metric topology.
		\end{enumerate}
	\end{theorem}

	\begin{remark}[Correspondences allowing the empty set]\label{RemarkBorelFieldEmptySet}
			We would like to also allow the empty set as a value of a correspondence. Denote by $\tau$ the exponential topology on $2^Y$.
			Following \cite[\S\,17]{KuratowskiTopology}, \cite{Michael1951TopologiesOS}, one can extend $\tau$ to a topology
			\begin{align*}
			\tau_\emptyset = \tau\cup\{B\cup\{\emptyset\} \ | \ B\in\tau\}
			\end{align*}
			on $2^Y_\emptyset=2^Y\cup\{\emptyset\}$, a basis of which is obviously given by \eqref{eqn:basisExponentialTopology} and the unions of \eqref{eqn:basisExponentialTopology} with $\{\emptyset\}$.
			The Borel $\sigma$-algebra $\mathcal{B}(\tau_\emptyset)$ generated by $\tau_\emptyset$ is then given by
			\begin{align*}
			\mathcal{B}(\tau_\emptyset)=\mathcal{B}(\tau)\cup\{B\cup\{\emptyset\} \ | \ B\in \mathcal{B}(\tau) \}.
			\end{align*}
			In essence, this allows us to prove measurability of a function with codomain $2^Y_\emptyset$ by considering preimages of sets $O\in\tau$ and of $\{\emptyset\}$ separately. The same holds for preimages of closed sets.
	\end{remark}
	
	For a metric space $(Y,d)$, weak measurability of $\varphi$ can in turn be reduced to a measurability and continuity condition on the associated distance function
	\begin{equation*}
	\delta_\varphi:X\times Y\to\R,\quad
	(x,y)\mapsto d(y,\varphi(x)).
	\end{equation*}
	To state this condition recall that a map $f:X\times Y\to\R$ is Carath\'eodory
	if $x\mapsto f(x,y)$ is measurable for every fixed $y$ and $y\mapsto f(x,y)$ is continuous for every fixed $x$.
	
	\begin{theorem}[{\cite[Thm.\,18.5]{InfDimAnalysis_Aliprantis}}]\label{thm:weakMeasurability}
		A non-empty-valued correspondence mapping a measurable space into a separable metrizable space is weakly measurable if and only if its associated distance function is Carath\'eodory.
	\end{theorem}
	
	With this preparation we can now state the desired measurability result.
	Below we abbreviate $(N,{\mathcal N})$ to be the codomain of point processes on $Z=[0,T]\times\boundDomDelta$
	and $Y=\measp(\spacetime)\times\meas(\spacetime)^3$ equipped with the weak-* topology.
	
	\begin{theorem}[Measurability of reconstruction mapping]
		The map $\varphi=\reconMap^{q,\beta,\halflife,T,\dom}$ is measurable from $({N},{\mathcal N})$ into $2^Y_\emptyset$ with the exponential topology.
	\end{theorem}
	\begin{proof}
		We stratify the space ${N}$ by setting ${N}_l=\{E\in{N}\,|\,|E|=l\}$ for $l\in\N_0\cup\{\infty\}$ (note that the ${N}_l$ are measurable).
		Consider first $l<\infty$. By \cref{thm:coercivity,thm:boundednessInfimum} there exists some $C>0$ such that any minimizer $(\rho,\momentum)$ of $J^{E,q}$ satisfies $\norm{\rho}+\norm{\momentum}<C(1+|E|)$.
		Thus, if we set $Y_l=\{(\rho,\momentum)\in Y\,|\,\norm{\rho}+\norm{\momentum}\leq C(1+l)\}\subset Y$, the restriction $\varphi_l$ of $\varphi$ to ${N}_l$ is a correspondence
		\begin{equation*}
		\varphi_l:{N}_l\twoheadrightarrow Y_l.
		\end{equation*}
		We now show that the distance function $\delta_{\varphi_l}$ associated with $\varphi_l$ is Carath\'eodory,
		then \cref{thm:weakMeasurability,ThmBorelMeasCompactSubspaces} imply the measurability of $\varphi_l$ with respect to the trace $\sigma$-algebra of ${\mathcal N}$
		since $Y_l$ (as a norm-ball of the dual to a separable Banach space) is separable and metrizable with respect to the weak-* topology.
		We denote the metric on $Y_l$ by $d$.
		We first show continuity in the second argument of $\delta_{\varphi_l}$:
		Fix $E\in{N}_l$ and let $(\rho_n,\momentum_n)\weakstarto(\rho,\momentum)$ in $Y_l$.
		By the weak-* compactness of $\varphi_l(E)=\varphi(E)$ from \cref{thm:ExistenceMinimizers} there is a sequence $(\tilde\rho_n,\tilde\momentum_n)\in\varphi_l(E)$
		with $\delta_{\varphi_l}(E,(\rho_n,\momentum_n))=d((\rho_n,\momentum_n),(\tilde\rho_n,\tilde\momentum_n))$.
		Furthermore, up to a subsequence (still indexed by $n$ for simplicity) we have $(\tilde\rho_n,\tilde\momentum_n)\weakstarto(\tilde\rho,\tilde\momentum)\in\varphi_l(E)$.
		Thus, for any $(\hat\rho,\hat\momentum)\in\varphi_l(E)$ we have
		\begin{equation*}
		d((\rho,\momentum),(\hat\rho,\hat\momentum))
		=\lim_{n\to\infty}d((\rho_n,\momentum_n),(\hat\rho,\hat\momentum))
		\geq\lim_{n\to\infty}\delta_{\varphi_l}(E,(\rho_n,\momentum_n))
		=\lim_{n\to\infty}d((\rho_n,\momentum_n),(\tilde\rho_n,\tilde\momentum_n))
		=d((\rho,\momentum),(\tilde\rho,\tilde\momentum))
		\end{equation*}
		so that $\delta_{\varphi_l}(E,(\rho,\momentum))=d((\rho,\momentum),(\tilde\rho,\tilde\momentum))$ and thus
		$\delta_{\varphi_l}(E,(\rho_n,\momentum_n))\to\delta_{\varphi_l}(E,(\rho,\momentum))$ for $n\to\infty$ as desired.
		We next show measurability of the map $E\mapsto\delta_{\varphi_l}(E,(\rho,\momentum))$ for fixed $(\rho,\momentum)$:
		To this end we first show sequential weak-* lower semicontinuity of that map, so let $E_n\weakstarto E$ in ${N}_l$ as $n\to\infty$
		and assume without loss of generality that $\liminf_{n\to\infty} \delta_{\varphi_l}(E_n,(\rho,\momentum))=\lim_{n\to\infty} \delta_{\varphi_l}(E_n,(\rho,\momentum))$
		(else we may pass to a subsequence).
		Since $\varphi_l(E_n)=\varphi(E_n)$ is compact by \cref{thm:ExistenceMinimizers} there exists a sequence $(\tilde\rho_n,\tilde\momentum_n)\in\varphi_l(E_n)\subset Y_l$
		such that $d((\rho,\momentum),\varphi_l(E_n))=d((\rho,\momentum),(\tilde\rho_n,\tilde\momentum_n))$.
		Then by \cref{thm:GammaConvergence} we can extract a subsequence (not relabled) such that $(\tilde\rho_n,\tilde\momentum_n)\to(\tilde\rho,\tilde\momentum)\in Y_l$ as well as $(\tilde\rho,\tilde\momentum)\in\varphi_l(E)$ and thus
		\begin{multline*}
		\delta_{\varphi_l}(E,(\rho,\momentum))
		=d((\rho,\momentum),\varphi_l(E))
		\leq d((\rho,\momentum),(\tilde\rho,\tilde\momentum))
		=\lim_{n\to\infty} d((\rho,\momentum),(\tilde\rho_n,\tilde\momentum_n))\\
		=\lim_{n\to\infty} d((\rho,\momentum),\varphi(E_n))
		=\lim_{n\to\infty} \delta_{\varphi_l}(E_n,(\rho,\momentum)),
		\end{multline*}
		proving the desired lower semi-continuity.
		This lower semicontinuity now implies measurability:
		Indeed, since the $\sigma$-algebra on ${N}_l$ is generated by the maps $E\mapsto\int f\,\d E$ for measurable $f\geq0$
		it contains the Borel $\sigma$-algebra of the weak-* topology, which is generated by the maps $E\mapsto\int u\,\d E$ for continuous functions $u\geq0$
		(we may restrict to nonnegative $u$ since ${N}_l$ only contains nonnegative measures). Applying \cref{ThmBorelMeasCompactSubspaces} we get Borel measurability of $\varphi_l$ as a mapping to $\mathcal{K}(Y_l)$.
		A consequence of this is the Borel measurability (in the sense of a map between measurable spaces, not of a correspondence) of the restriction
		\begin{equation*}
		\varphi_{<\infty}:{N}_{<\infty}\rightarrow 2^Y\subset 2^Y_\emptyset
		\end{equation*}
		of $\varphi$ to ${N}_{<\infty}$.
		Indeed, if $G\subset2^Y$ is closed in $2^Y$, then $G_l=G\cap2^{Y_l}$ is closed in $2^{Y_l}$ as we have argued in the introduction of the appendix. By compactness of $Y_l$ it holds $2^{Y_l}=\mathcal{K}(Y_l)$ and hence $G_l$ is closed in $\mathcal{K}(Y_l)$.
		The Borel measurability of $\varphi_l$ (as a map into $\mathcal{K}(Y_l)$) then implies measurability of $\varphi_l^{-1}(G_l)$ with respect to the trace $\sigma$-algebra on ${N}_l$
		and due to the measurability of ${N}_l$ also with respect to ${\mathcal N}$.
		Thus we obtain measurability of
		\begin{equation*}
		\varphi_{<\infty}^{-1}(G)
		=\bigcup_{l\in\N}\varphi_l^{-1}(G)
		=\bigcup_{l\in\N}\varphi_l^{-1}(G_l).
		\end{equation*}
		Finally, due to $\varphi^{-1}(\{\emptyset\})={N}_\infty$, all of $\varphi$ is measurable (see \cref{RemarkBorelFieldEmptySet} and note that $\{\emptyset\}$ is closed in $2^Y_\emptyset$).
	\end{proof}

\end{appendices}

\bibliographystyle{plain}
\bibliography{99_bibliography}

\begin{thebibliography}{10}

\bibitem{defVolumeMeasure}
S~Albeverio, Yu.G Kondratiev, and M~Röckner.
\newblock Analysis and geometry on configuration spaces.
\newblock {\em Journal of Functional Analysis}, 154(2):444--500, 1998.

\bibitem{InfDimAnalysis_Aliprantis}
Charalambos~D. Aliprantis and Kim~C. Border.
\newblock {\em Infinite dimensional analysis}.
\newblock Springer, Berlin, third edition, 2006.
\newblock A hitchhiker's guide.

\bibitem{AmFuPa00}
Luigi Ambrosio, Nicola Fusco, and Diego Pallara.
\newblock {\em Functions of bounded variation and free discontinuity problems}.
\newblock Oxford Mathematical Monographs. The Clarendon Press, Oxford
  University Press, New York, 2000.

\bibitem{Be21}
Jean-David Benamou.
\newblock Optimal transportation, modelling and numerical simulation.
\newblock {\em Acta Numerica}, 30:249--325, 2021.

\bibitem{BeBr00}
Jean-David Benamou and Yann Brenier.
\newblock A computational fluid mechanics solution to the {M}onge-{K}antorovich
  mass transfer problem.
\newblock {\em Numerische Mathematik}, 84(3):375--393, 2000.

\bibitem{Bredies_extremePointsBB}
Kristian Bredies, Marcello Carioni, Silvio Fanzon, and Francisco Romero.
\newblock On the extremal points of the ball of the {B}enamou--{B}renier
  energy.
\newblock {\em Bulletin of the London Mathematical Society}, 53(5):1436--1452,
  2021.

\bibitem{ChPeScVi18}
L.~Chizat, G.~Peyr\'e, B.~Schmitzer, and F.-X. Vialard.
\newblock Unbalanced optimal transport: Dynamic and {Kantorovich} formulations.
\newblock {T}o appear in J.~Funct.~Anal., arXiv:1508.05216, 2018.

\bibitem{chizat_unbalancedOT}
Lena{\"i}c Chizat.
\newblock {\em {Unbalanced Optimal Transport : Models, Numerical Methods,
  Applications}}.
\newblock Theses, {Universit{\'e} Paris sciences et lettres}, November 2017.

\bibitem{EvGa15}
Lawrence~C. Evans and Ronald~F. Gariepy.
\newblock {\em Measure theory and fine properties of functions}.
\newblock Textbooks in Mathematics. CRC Press, Boca Raton, FL, revised edition,
  2015.

\bibitem{FedererGMT}
Herbert Federer.
\newblock {\em Geometric Measure Theory}.
\newblock Springer, Berlin, Heidelberg, 1996.

\bibitem{KlenkeProbabilityTheory}
Achim Klenke.
\newblock {\em Probability Theory}.
\newblock Springer Berlin Heidelberg, 2013.

\bibitem{KuratowskiTopology}
K.~Kuratowski.
\newblock {\em Topology.}
\newblock New York: Academic Press. 2 vols, 1966.

\bibitem{bookPPP_last_penrose}
Günter Last and Mathew Penrose.
\newblock {\em Lectures on the Poisson Process}.
\newblock Institute of Mathematical Statistics Textbooks. Cambridge University
  Press, 2017.

\bibitem{Lee2015}
Keum~Sil Lee, Tae~Jin Kim, and Guillem Pratx.
\newblock {Single-Cell Tracking With PET Using a Novel Trajectory
  Reconstruction Algorithm}.
\newblock {\em IEEE Trans. Med. Imaging}, 34(4):994--1003, apr 2015.

\bibitem{Michael1951TopologiesOS}
Ernest Michael.
\newblock Topologies on spaces of subsets.
\newblock {\em Transactions of the American Mathematical Society}, 71:152--182,
  1951.

\bibitem{ReissPP}
R.-D. Reiss.
\newblock {\em A Course on Point Processes}.
\newblock 1993.

\bibitem{OTAppliedMath}
Filippo Santambrogio.
\newblock {\em Optimal Transport for Applied Mathematicians}.
\newblock Birkhäuser Basel, 2015.

\bibitem{ScScWi20}
Bernhard Schmitzer, Klaus Sch\"afers, and Benedikt Wirth.
\newblock Dynamic cell imaging in pet with optimal transport regularization.
\newblock {\em IEEE Transactions on Medical Imaging}, 39(5):1626--1635, 2020.

\end{thebibliography}
\end{document}